\documentclass[12pt]{elsarticle}

\usepackage{amsmath, amsthm, amssymb}
\usepackage{bm}                  
\usepackage{mathtools}          
\usepackage{hyperref}

\newtheorem{theorem}{Theorem}[section]
\newtheorem{proposition}[theorem]{Proposition}
\newtheorem{corollary}[theorem]{Corollary}
\newtheorem{lemma}[theorem]{Lemma}
\newtheorem{remark}{Remark}[section]
\newtheorem{example}{Example}[section]
\newtheorem{definition}{Definition}[section]

\newcommand{\R}{\mathbb{R}}
\newcommand{\ones}{\bm{1}}              
\newcommand{\bg}{\bm{g}}                
\newcommand{\bh}{\bm{h}}                
\newcommand{\bw}{\bm{w}}                
\newcommand{\bx}{\bm{x}}                
\newcommand{\normF}[1]{\|#1\|_F}
\newcommand{\normtwo}[1]{\|#1\|_2}
\newcommand{\rank}{\operatorname{rank}}
\newcommand{\tr}{\operatorname{tr}}
\newcommand{\diag}{\operatorname{diag}}
\newcommand{\Symn}[1][n]{\mathbb{S}^{#1}}
\newcommand{\bu}{\bm{u}}
\newcommand{\by}{\bm{y}}
\newcommand{\be}{\bm{e}}
\newcommand{\bv}{\bm{v}}
\newcommand{\Dtwo}{D^{(2)}}
\newcommand{\gF}{\bg^F}
\newcommand{\normui}[1]{\left|\!\left|\!\left|#1\right|\!\right|\!\right|}  
\newcommand{\CompI}[2]{\mathrm{Comp}_{#1}(#2)}                

\begin{document}

\begin{frontmatter}

\title{Universal optimality of the double-centred matrix under
       unitarily invariant norms for dissimilarity data}

\author[uchceu]{M.~Nuria de las Heras Santos}
\ead{marianuria.deherassantos@alumnos.uchceu.es}

\author[uchceu]{Antonio Falc\'o\corref{cor1}}
\ead{afalco@uchceu.es}

\author[uchceu]{Francisco Javier Mu\~{n}oz Almaraz}
\ead{malmaraz@uchceu.es}

\cortext[cor1]{Corresponding author.}
\address[uchceu]{Departamento de Matem\'aticas, F\'{\i}sica y Ciencias Tecnol\'ogicas,
Universidad Cardenal Herrera-CEU, CEU Universities,
Calle San Bartolom\'e, 55, Alfara del Patriarca 46115, Valencia, Spain}

\begin{abstract}
    The classical identity relating squared distances to inner products
    determines every off-diagonal entry of a symmetric matrix associated
    with a set of pairwise dissimilarities, while leaving its diagonal
    free.  This paper studies the resulting affine
    family of symmetric matrices and asks which representative is
    selected by minimising a matrix norm.  We prove that the same representative,
    namely the double-centred matrix used in classical multidimensional
    scaling and principal coordinate analysis, minimises every unitarily
    invariant norm on the family.  The result includes the Frobenius,
    spectral, nuclear, Schatten, and Ky Fan norms, and requires no
    assumption that the dissimilarities are Euclidean distances.  The
    proof rests on a compression-reduction principle: the centring projection maps
    the entire compatible set to a single element of the family, and
    orthogonal compression is contractive for unitarily invariant norms;
    the conclusion follows without solving a separate minimisation
    problem for each norm.
    We then determine how many minimisers
    there are: the minimiser is unique for the Frobenius norm, for the
    nuclear norm and for every strictly convex unitarily invariant norm,
    while for the spectral norm we describe the whole minimising set, a
    compact convex set that reduces to a single point precisely when the
    double-centred matrix vanishes.
    \end{abstract}


\end{frontmatter}

\section{Introduction}
\label{sec:intro}

Distance and dissimilarity matrices arise throughout geometry,
multivariate analysis, numerical ecology, and data science; see,
for example, \cite{Gower1966,Dokmanic2015,Legendre2012,Deza1997}.
In these settings one records pairwise dissimilarities between
$n$ objects in a symmetric matrix $D=(D_{ij})_{i,j=1}^{n}$ with
$D_{ii}=0$, an object being at zero dissimilarity from itself.  We assume
throughout that $n\geq 2$; beyond symmetry and this diagonal normalisation
nothing else is assumed: the arguments below
use only the off-diagonal entries of the squared dissimilarity matrix
$\Dtwo=(D_{ij}^{2})$, so neither non-negativity, nor the triangle
inequality, nor Euclidean realisability is imposed.

\subsection*{The algebraic structure of squared dissimilarity data}

A fundamental relation connecting squared dissimilarity data to
inner-product structure is the identity
\begin{equation}\label{eq:gram}
  D_{ij}^{2} = a_{ii}+a_{jj}-2a_{ij}, \qquad i,j=1,\dots,n,
\end{equation}
where $A=(a_{ij})$ is a symmetric matrix.  In the classical
Euclidean setting, where $D$ is a distance matrix and hence
$D_{ii}=0$, this identity is also consistent for $i=j$ and expresses
the squared distance between two points in terms of their Gram matrix.
It was brought to prominence by Young and
Householder~\cite{YoungHouseholder1938} and, in the form of
Schoenberg's criterion~\cite{Schoenberg1938}, it decides when a
symmetric matrix of squared distances comes from a Euclidean
configuration: the test is the positive semidefiniteness of the
double-centred matrix $-\tfrac{1}{2}J\Dtwo J$, where
$J=I-\tfrac{1}{n}\ones\ones^{\top}$ is the centring projection.  The
same double-centred matrix is the object diagonalised by classical
multidimensional scaling (CMDS), known in numerical ecology and related
applied fields as principal coordinate analysis
(PCoA)~\cite{Torgerson1952,Gower1966}.  These statements, together with
the scaling procedure itself and references to the surveys
\cite{Krislock2012,Liberti2014,Dokmanic2015}, are collected in
Section~\ref{subsec:background}; see in particular
Theorem~\ref{thm:schoenberg}.

In the present paper, we do not assume that the dissimilarity matrix
comes from a Euclidean configuration.  We only prescribe the
off-diagonal part of \eqref{eq:gram}, namely
\[
  D_{ij}^{2} = a_{ii}+a_{jj}-2a_{ij},
  \qquad i\neq j.
\]
Once a free vector $\bg=(g_{1},\dots,g_{n})^{\top}\in\R^{n}$ is fixed,
these relations determine all off-diagonal entries of $A$.  The
corresponding affine family is
\begin{equation}\label{eq:Ag}
  A(\bg)
  =
  \tfrac{1}{2}\!\left(\ones \bg^{\top}
  +\bg\ones^{\top}-\Dtwo\right),
  \qquad \bg\in\R^{n}.
\end{equation}
Indeed, for $i\neq j$ this gives
\[
  a_{ij}(\bg)
  =
  \tfrac{1}{2}\bigl(g_i+g_j-D_{ij}^{2}\bigr),
\]
so that $D_{ij}^{2}=a_{ii}+a_{jj}-2a_{ij}$ for $i\neq j$, while
$a_{ii}(\bg)=g_i$.  Thus $\bg$ is exactly the diagonal of $A(\bg)$, and
\eqref{eq:Ag} parametrises all symmetric matrices satisfying
\eqref{eq:gram} for every pair $i,j$: the squared dissimilarities fix
the off-diagonal structure and leave the diagonal free.

The selection problem considered here is therefore simple to state:
among all symmetric matrices compatible with the prescribed squared
dissimilarity data, choose the one of smallest size.  Different matrix
norms measure different notions of size.  The Frobenius norm measures
the total entrywise energy, the spectral norm measures the largest
operator effect, and the nuclear and Ky Fan norms depend on different
aggregates of singular values.  A priori these criteria need not select
the same representative of the affine family.  The main point of this
paper is that, for the family \eqref{eq:Ag}, they do.

\subsection*{The selection problem and its relation to existing work}

The problem considered here may be viewed as an affine variant of a
\emph{matrix nearness problem} in the sense of
Higham~\cite{Higham1988,Higham1989}: one seeks, with respect to a
prescribed matrix norm, a nearest matrix satisfying a specified
structural constraint.  
The literature on matrix nearness problems is extensive; the problem of
approximating an arbitrary matrix by a nearest symmetric positive
semidefinite matrix in the Frobenius and spectral norms was studied in
detail by Higham~\cite{Higham1988}.  In the Frobenius norm, the nearest
positive-semidefinite matrix is unique and has an explicit formula in
terms of the polar decomposition, whereas in the spectral norm, nearest
positive-semidefinite matrices need not be unique, and their computation
is based on Halmos' scalar characterisation, leading in general to
iterative algorithms.  Thus, the two norm-minimisation problems may select
different approximants.
The contrast with the present paper is instructive: here we show that,
for the specific affine family \eqref{eq:Ag}, the two norm-minimisation
problems are solved by the \emph{same} matrix and that a closed-form
expression exists.  This common representative is the double-centred
matrix already used in CMDS and PCoA; the new contribution is its
norm-independent variational characterisation over the affine family.

A related line of work concerns the completion of partial distance
matrices.  Given a partial symmetric matrix with only certain entries
specified, the \emph{Euclidean distance matrix completion problem}
(EDMCP) asks for the unspecified entries that make the full matrix a
valid EDM.  Alfakih, Khandani, and Wolkowicz~\cite{Alfakih1999} studied the EDMCP
through a semidefinite-programming formulation of an approximate
completion problem.  Their formulation relies on the classical
Gram-to-distance map, in which the diagonal of the underlying Gram
matrix enters the distance identities in a way analogous to the free
vector \(\bg\) in \eqref{eq:Ag}. The key difference is that their
optimisation is carried out over positive semidefinite centred Gram
matrices, equivalently over Euclidean distance matrices, whereas the
present paper imposes no Euclidean realisability assumption on the
prescribed dissimilarity data and selects a canonical representative of
the affine family \eqref{eq:Ag} by norm minimisation.  
Trosset~\cite{Trosset2000} and Fang and O'Leary~\cite{FangOLeary2012} 
developed further numerical methods for distance matrix completion; 
again, these works focus on finding EDM completions, not on identifying 
the canonical representative of the unconstrained affine family \eqref{eq:Ag}.

In many applications of current interest—microbiota
analysis~\cite{Satten2017} and numerical ecology~\cite{Legendre2012}—the
dissimilarity matrix $D$ may fail to be a metric and need not arise
from any Euclidean embedding.  
The coefficients may be chosen for biological or ecological reasons 
(e.g., Bray–Curtis dissimilarity, UniFrac), and the classical EDM theory 
is therefore not directly applicable.  It then becomes natural to study 
the affine family \eqref{eq:Ag} without imposing realisability, and to
ask whether the free parameter can be selected in a principled way.

This is precisely the situation in which PCoA is often used in practice:
one starts from a dissimilarity matrix chosen for scientific reasons, even
when the associated double-centred matrix need not be positive semidefinite.
The present paper therefore provides a variational explanation for the same
double-centred matrix used in CMDS/PCoA, but without assuming that the input
dissimilarities are Euclidean.

\subsection*{Universal optimality and the double-centred matrix}

The class of \emph{unitarily invariant norms} provides the natural
framework for the present work.  This class includes the
\emph{Frobenius norm} $\normF{A} = (\sum_{i,j}a_{ij}^{2})^{1/2}$,
a global quadratic criterion measuring the total energy of the matrix
entries, and the \emph{spectral norm} $\normtwo{A} = \sigma_1(A)$,
equal to the largest singular value, which captures the extremal
operator-theoretic behaviour and is a standard norm in
perturbation theory and numerical analysis~\cite{Higham1988,Horn2013}.  Both, together with all Schatten
$p$-norms, the nuclear norm, and the Ky Fan $k$-norms, are
characterised by the singular values via symmetric gauge functions
(Definition~\ref{def:gauge} and Theorem~\ref{thm:vonneumann};
see also \cite{Horn2013,Bhatia1997}).  Since these norms quantify
genuinely different aspects of matrix size, there is no a priori reason
why the minimisation problems
\[
  \min_{\bg\in\R^n}\normui{A(\bg)}
\]
should all be solved by the \emph{same} element of the affine
family~\eqref{eq:Ag}.  The main result of this paper shows that they
are, and identifies that common minimiser as the double-centred matrix
$-\tfrac{1}{2}J\Dtwo J$.

\subsection*{Contribution and novelty}

The present paper makes the following contributions:
\begin{enumerate}
  \item \textbf{Main result.}  We prove that the double-centred matrix
        $A(\gF)=-\tfrac{1}{2}J\Dtwo J$ is a minimiser of
        $\normui{A(\bg)}$ for \emph{every} unitarily invariant norm
        $\normui{\cdot}$ (Theorem~\ref{thm:main}).  This includes all
        Schatten $p$-norms ($1\leq p\leq\infty$), the nuclear norm, and
        the Ky Fan $k$-norms.  The proof rests on a compression-reduction
        principle (Definition~\ref{def:compclosed} and
        Lemma~\ref{lem:compreduction}): the affine family in
        \eqref{eq:Ag} is compression-closed under the centring
        projection $J$ with singleton image $\{A(\gF)\}$
        (Lemmas~\ref{lem:compression} and~\ref{lem:gF}), and orthogonal
        compression is contractive for every unitarily invariant norm
        (Lemma~\ref{lem:proj}, following \cite{Bhatia1997,Bhatia2000}).
        The argument requires no variational analysis: once $A(\gF)$ is
        identified as the centred representative, universal optimality
        follows without solving a separate minimisation problem for each
        norm.
        This optimal representative is unique for strictly convex
        unitarily invariant norms, in particular for the Schatten-\(p\)
        norms with \(1<p<\infty\), including the Frobenius norm.  We also
        prove uniqueness for the nuclear norm.  For the spectral norm
        uniqueness fails whenever $A(\gF)\neq 0$, and we determine the
        whole set of spectral minimisers
        (Theorem~\ref{thm:spectral_set}).  For the Ky Fan \(k\)-norms with
        \(1<k<n\) the description of the minimising sets is left open.
  \item We prove that $\min_{\bg}\normF{A(\bg)}$ admits a unique
        minimiser $\gF$ given by an explicit closed formula
        (Lemma~\ref{lem:gF}), which provides a new \emph{variational characterisation} of
        the double-centred matrix central to CMDS and to its applied 
        formulation as PCoA \cite{Torgerson1952,Gower1966}: it is 
        the unique representative of \eqref{eq:Ag} selected by the 
        Frobenius norm, and a minimiser of every unitarily invariant 
        norm simultaneously, without any Euclidean realisability 
        assumption.      
        We further characterise, via a rank condition on $J\Dtwo J$,
        when all Schatten norms coincide at the optimum
        (Corollary~\ref{cor:rank} and
        Proposition~\ref{prop:rank1geom}): this occurs if and only
        if the double-centred squared dissimilarities factor as a
        rank-at-most-one outer product, which in the Euclidean case is
        equivalent to collinearity of the $n$ points.
  \item Along the natural one-dimensional slice
        $\bg(t)=\gF+\frac{t}{n}\ones$, we obtain the exact formula
        \[
            \normtwo{A(\bg(t))}=\max\{\rho^*,|t|\},
            \qquad
            \rho^*=\normtwo{A(\gF)}.
        \]
        Consequently, whenever $\rho^*>0$, the spectral-norm minimiser
        is not unique: every $t$ with $|t|\leq \rho^*$ gives another
        spectral-norm minimiser along this slice.  By contrast, for
        strictly convex unitarily invariant norms, in particular for
        the Schatten-$p$ norms with $1<p<\infty$, the minimiser is
        unique and is given by $\gF$.  Theorem~\ref{thm:spectral_set}
        goes further and characterises \emph{all} spectral minimisers:
        after passing to coordinates in which the affine family becomes
        the set of symmetric matrices with a prescribed compression to
        $\ones^{\perp}$ and a free border, the minimisers are exactly the
        borders whose vector part is orthogonal to the extremal
        eigenvectors of $A(\gF)$ and whose scalar part lies in an
        explicitly computable interval.  The slice above is the special
        case of a vanishing vector part.  The corresponding question for
        the Ky Fan $k$-norms with $1<k<n$ is left open.
\end{enumerate}

\subsection*{Organisation}

Section~\ref{sec:prelim} fixes the notation, defines the affine family
and the classes of norms involved, and collects the classical background
on Euclidean distance matrices, Schoenberg's criterion and classical
scaling.  Section~\ref{sec:main} states and proves the main results.
Section~\ref{sec:example} provides explicit numerical illustrations.
Section~\ref{sec:discussion} discusses connections with CMDS,
extensions to other unitarily invariant norms, and open questions.
A reader familiar with Euclidean distance matrices and with von
Neumann's theorem may go directly from Section~\ref{sec:intro} to
Section~\ref{sec:main}.

\section{Notation, preliminaries, and background}
\label{sec:prelim}

This section fixes the notation, gives precise definitions of the
objects used throughout the paper, and collects the classical facts on
which the argument relies.  Nothing in this section is new; it is
included so that the statements of Section~\ref{sec:main} can be read
without recourse to outside sources.

\subsection{Notation}
\label{subsec:notation}

We write $\R^{n\times n}$ for the space of real $n\times n$ matrices and
$\Symn$ for its subspace of symmetric matrices.  Vectors are columns,
$I$ is the identity matrix and $\ones=(1,\dots,1)^{\top}\in\R^{n}$ is
the all-ones vector, whose orthogonal complement is denoted
$\ones^{\perp}$.  For $A\in\R^{n\times n}$ we write $a_{ij}$ for its
entries, $A^{\top}$ for its transpose, $\tr A$ for its trace,
$\rank A$ for its rank and $\diag(A)\in\R^{n}$ for the vector of its
diagonal entries; $A\succeq 0$ means that $A$ is symmetric positive
semidefinite.  The singular values of $A$ are written
$\sigma(A)=(\sigma_{1}(A),\dots,\sigma_{n}(A))$, always in
non-increasing order; for $A\in\Symn$ the eigenvalues are
$\lambda_{1}(A),\dots,\lambda_{n}(A)$, and the singular values are their
absolute values arranged in non-increasing order.  The Frobenius inner
product is $\langle A,B\rangle_{F}=\tr(A^{\top}B)$, with associated norm
$\normF{A}=\langle A,A\rangle_{F}^{1/2}$.  On $\R^{n}$ we write
$\|\bx\|$ for the Euclidean norm and $\|\bx\|_{p}$ for the $\ell^{p}$
norm, $1\leq p\leq\infty$.

\subsection{Dissimilarity data and the affine family}
\label{subsec:family}

\begin{definition}[Dissimilarity data]
    \label{def:dissim}
    A \emph{dissimilarity matrix} on $n$ objects is a symmetric matrix
    $D\in\Symn$ with $D_{ii}=0$ for every $i$; the entry $D_{ij}$ is the
    dissimilarity between objects $i$ and $j$.  Its \emph{squared
    dissimilarity matrix} is
    $\Dtwo=(D_{ij}^{2})_{i,j=1}^{n}\in\Symn$.  If moreover
    $D_{ij}\geq 0$ for all $i,j$, then $D$ is a \emph{distance
    matrix}; and $D$ is a \emph{Euclidean distance matrix} (EDM) if
    there exist an integer $d\geq 0$ and points
    $\bx_{1},\dots,\bx_{n}\in\R^{d}$ with
    $D_{ij}=\|\bx_{i}-\bx_{j}\|$ for all $i,j$.
\end{definition}

The standing assumption of this paper is that $n\geq 2$ and $D$ is
symmetric with zero diagonal.  Neither the triangle inequality, nor non-negativity, nor
Euclidean realisability is assumed anywhere, and the data actually used
are the off-diagonal entries of $\Dtwo$.  The terms \emph{distance
matrix} and
\emph{EDM} are reserved for the situations described in
Definition~\ref{def:dissim}, and are used only when the corresponding
hypothesis is stated explicitly.

\begin{definition}[Compatible matrices]
    \label{def:compatible}
    A matrix $A=(a_{ij})\in\Symn$ is \emph{compatible} with $\Dtwo$ if
    \begin{equation}\label{eq:compat}
      D_{ij}^{2}=a_{ii}+a_{jj}-2a_{ij}
      \qquad\text{for all } i\neq j .
    \end{equation}
    The set of matrices compatible with $\Dtwo$ is denoted
    $\mathcal{C}(\Dtwo)$.
\end{definition}

Condition \eqref{eq:compat} constrains the off-diagonal entries of $A$
only; it says nothing about $\diag(A)$, and it is exactly this
under-determination that the paper is about.  The next proposition makes
the parametrisation of $\mathcal{C}(\Dtwo)$ precise.

\begin{proposition}[Parametrisation of the compatible set]
    \label{prop:param}
    The map $\bg\mapsto A(\bg)$ defined by \eqref{eq:Ag} is an affine
    bijection from $\R^{n}$ onto $\mathcal{C}(\Dtwo)$, with inverse
    $A\mapsto\diag(A)$.  In particular $\mathcal{C}(\Dtwo)$ is an affine
    subspace of $\Symn$ of dimension $n$.
\end{proposition}

\begin{proof}
    The map is affine because $\bg\mapsto\ones\bg^{\top}+\bg\ones^{\top}$
    is linear.  If $A(\bg)=A(\bg')$ then, with $\bu=\bg-\bg'$,
    \[
      \ones\bu^{\top}+\bu\ones^{\top}=0 ,
    \]
    and the $(i,i)$ entry of the left-hand side is $2u_{i}$, so $\bu=0$.
    Hence the map is injective and its image is an affine subspace of
    dimension $n$.

    From \eqref{eq:Ag} and $D_{ii}=0$ we have
    $a_{ii}(\bg)=g_{i}$ and $a_{ij}(\bg)=\tfrac{1}{2}(g_{i}+g_{j}
    -D_{ij}^{2})$ for $i\neq j$, whence
    $a_{ii}(\bg)+a_{jj}(\bg)-2a_{ij}(\bg)=D_{ij}^{2}$; thus every
    $A(\bg)$ lies in $\mathcal{C}(\Dtwo)$ and $\bg=\diag(A(\bg))$.
    Conversely, let $A\in\mathcal{C}(\Dtwo)$ and put $\bg=\diag(A)$.
    For $i\neq j$, \eqref{eq:compat} gives
    $a_{ij}=\tfrac{1}{2}(g_{i}+g_{j}-D_{ij}^{2})=a_{ij}(\bg)$, while the
    diagonals of $A$ and $A(\bg)$ agree by the choice of $\bg$.  Hence
    $A=A(\bg)$.
\end{proof}

\begin{remark}[Dropping the normalisation $D_{ii}=0$]
    \label{rem:diag}
    None of the results below uses the zero diagonal directly: every
    proof uses only that the family under consideration is the affine
    set \eqref{eq:Ag}, and \eqref{eq:Ag} makes sense for an arbitrary
    symmetric $D$.  All the statements of Section~\ref{sec:main}
    therefore remain valid for an arbitrary symmetric~$D$, provided
    $\mathcal{C}(\Dtwo)$ in each statement is read as the affine
    family \eqref{eq:Ag}.

    What changes is the interpretation, and this is why we prefer to
    normalise.  For arbitrary $\diag(D)$ one has
    $a_{ii}(\bg)=g_{i}-\tfrac{1}{2}D_{ii}^{2}$, so $\bg$ is a shifted
    diagonal rather than the diagonal itself; moreover the resulting
    family is the translate of $\mathcal{C}(\Dtwo)$ by the fixed
    diagonal matrix $\Delta=-\tfrac{1}{2}\operatorname{diag}(D_{11}^{2},
    \dots,D_{nn}^{2})$, and this translate meets $\mathcal{C}(\Dtwo)$
    only when $\Delta=0$.  Indeed, if
    $\Delta=\tfrac{1}{2}(\ones\bh^{\top}+\bh\ones^{\top})$ for some
    $\bh\in\R^{n}$, then comparing diagonal entries gives
    $h_{i}=-\tfrac{1}{2}D_{ii}^{2}\leq 0$, while the vanishing of the
    off-diagonal entries gives $h_{i}=-h_{j}$ for $i\neq j$; the two
    together force $\bh=0$.  Under the
    normalisation of Definition~\ref{def:dissim} the family and the
    compatible set coincide, and Proposition~\ref{prop:param} is an
    equality of sets rather than a statement about a translate.
\end{remark}

\begin{definition}[Centring projection and double-centred matrix]
    \label{def:centring}
    The \emph{centring projection} is
    \[
      J=I-\tfrac{1}{n}\ones\ones^{\top} .
    \]
    It satisfies $J=J^{\top}=J^{2}$ and $J\ones=0$, so it is the
    orthogonal projection of $\R^{n}$ onto $\ones^{\perp}$.  The
    \emph{double-centred matrix} associated with $D$ is
    $-\tfrac{1}{2}J\Dtwo J$.
\end{definition}

\begin{definition}[Orthogonal compression]
    \label{def:compression}
    Let $P\in\R^{n\times n}$ be the orthogonal projection onto a
    subspace $\mathcal{S}\subseteq\R^{n}$.  For $M\in\R^{n\times n}$,
    the matrix $PMP\in\R^{n\times n}$ is the
    \emph{orthogonal compression of $M$ to $\mathcal{S}$}
    (represented in the ambient space, with zeros on
    $\mathcal{S}^{\perp}$).  If $Q\in\R^{n\times r}$ has
    orthonormal columns spanning $\mathcal{S}$, the matrix
    $Q^{\top}MQ\in\R^{r\times r}$ is the compression in the
    coordinates given by $Q$.
\end{definition}

\subsection{Classical background: Gram matrices, Schoenberg's criterion,
and classical scaling}
\label{subsec:background}

The identity \eqref{eq:compat} is the algebraic form of the relation
between squared distances and inner products.  If
$X=[\bx_{1}|\cdots|\bx_{n}]\in\R^{d\times n}$ and $A=X^{\top}X$ is the
Gram matrix of the columns of $X$, then
\begin{equation}\label{eq:gram_identity}
  \|\bx_{i}-\bx_{j}\|^{2}=a_{ii}+a_{jj}-2a_{ij},
  \qquad i,j=1,\dots,n .
\end{equation}
Young and Householder~\cite{YoungHouseholder1938} observed that
\eqref{eq:gram_identity} can be read in the opposite direction, and used
it to decide when a symmetric matrix of squared distances comes from a
Euclidean configuration.  Schoenberg~\cite{Schoenberg1938} characterised
Hilbertian metrics through conditional negative definiteness of squared
distances; in the finite-dimensional case the two results combine into
the following statement, which is the form used in this paper.

\begin{theorem}[Young--Householder; Schoenberg]
    \label{thm:schoenberg}
    Let $D$ be a distance matrix in the sense of
    Definition~\ref{def:dissim}.  Then $D$ is a Euclidean distance
    matrix if and only if
    \[
      -\tfrac{1}{2}J\Dtwo J\succeq 0 .
    \]
    In that case the smallest dimension $d$ admitting a configuration
    $\bx_{1},\dots,\bx_{n}\in\R^{d}$ with $D_{ij}=\|\bx_{i}-\bx_{j}\|$
    equals $\rank\bigl(-\tfrac{1}{2}J\Dtwo J\bigr)$, and such a
    configuration is obtained from any factorisation
    $-\tfrac{1}{2}J\Dtwo J=X^{\top}X$ with $X\in\R^{d\times n}$.
    See \cite{YoungHouseholder1938,Schoenberg1938} and, for modern
    accounts, \cite{Krislock2012,Liberti2014,Dokmanic2015}.
\end{theorem}

Theorem~\ref{thm:schoenberg} is the basis of \emph{classical
multidimensional scaling} (CMDS), introduced in this form by
Torgerson~\cite{Torgerson1952} and subsequently formulated by
Gower~\cite{Gower1966} as \emph{principal coordinate analysis} (PCoA);
the method is now widely used in numerical ecology~\cite{Legendre2012} and related fields.  The procedure is the following: given $D$,
form the double-centred matrix $-\tfrac{1}{2}J\Dtwo J$, compute its
spectral decomposition
\[
  -\tfrac{1}{2}J\Dtwo J=\sum_{k=1}^{n}\lambda_{k}\bu_{k}\bu_{k}^{\top},
  \qquad \lambda_{1}\geq\cdots\geq\lambda_{n},
\]
retain the $r$ positive eigenvalues, and take as coordinates of object
$i$ the vector
\[
  \bigl(\sqrt{\lambda_{1}}\,u_{1i},\dots,\sqrt{\lambda_{r}}\,u_{ri}\bigr)^{\top}\in\R^{r}.
\]
When $D$ is a
Euclidean distance matrix this reproduces the interpoint distances
exactly; when it is not, the negative eigenvalues are discarded and the
output is an approximation whose quality is judged by the discarded
part of the spectrum.

The double-centred matrix is thus the object that CMDS and PCoA
diagonalise, and in the applied literature it is normally introduced as
a recipe: centre, then decompose.  The present paper gives it a
different status.  It is not assumed to be positive semidefinite, and it
is not selected by a Euclidean argument; it is characterised as the
solution of a matrix-norm minimisation problem over
$\mathcal{C}(\Dtwo)$, simultaneously for every unitarily invariant norm.
This is relevant precisely in the applied situations where the recipe is
used but the theory behind it does not apply: in microbiota
analysis~\cite{Satten2017} and numerical ecology~\cite{Legendre2012},
the coefficients are chosen for scientific reasons
(Bray--Curtis dissimilarity, UniFrac) and the resulting $D$ is
frequently neither a metric nor Euclidean.

\subsection{Unitarily invariant norms}
\label{subsec:uinorms}

\begin{definition}[Unitarily invariant norm]
    \label{def:uinorm}
    A norm $\normui{\cdot}$ on $\R^{n\times n}$ is \emph{unitarily
    invariant} if $\normui{UAV}=\normui{A}$ for all $A\in\R^{n\times n}$
    and all real orthogonal $U,V\in\R^{n\times n}$.
\end{definition}

\begin{definition}[Symmetric gauge function]
    \label{def:gauge}
    A function $\Phi\colon\R^{n}\to\R$ is a \emph{symmetric gauge
    function} if it is a norm on $\R^{n}$ and
    \[
      \Phi(x_{1},\dots,x_{n})
      =
      \Phi(\varepsilon_{1}x_{\pi(1)},\dots,\varepsilon_{n}x_{\pi(n)})
    \]
    for every permutation $\pi$ of $\{1,\dots,n\}$ and every choice of
    signs $\varepsilon_{1},\dots,\varepsilon_{n}\in\{-1,1\}$; that is,
    if it is invariant under permutations and sign changes of the
    coordinates.
\end{definition}

\begin{theorem}[von Neumann; see {\cite[Ch.~IV]{Bhatia1997}} and
    {\cite{Horn2013}}]
    \label{thm:vonneumann}
    A norm $\normui{\cdot}$ on $\R^{n\times n}$ is unitarily invariant
    if and only if there is a symmetric gauge function $\Phi$ with
    \[
      \normui{A}=\Phi\bigl(\sigma(A)\bigr)
      \qquad\text{for all } A\in\R^{n\times n}.
    \]
\end{theorem}

The norms used below are the following special cases, listed with their
gauge functions.

\begin{itemize}
  \item the \emph{Frobenius norm}
        $\normF{A}=\bigl(\sum_{i,j}a_{ij}^{2}\bigr)^{1/2}$
        $=\bigl(\sum_{i}\sigma_{i}(A)^{2}\bigr)^{1/2}$, with
        $\Phi(x)=\|x\|_{2}$;
  \item the \emph{spectral norm} $\normtwo{A}=\sigma_{1}(A)$, with
        $\Phi(x)=\|x\|_{\infty}$;
  \item the \emph{Schatten $p$-norms}
        $\|A\|_{S_{p}}=\bigl(\sum_{i}\sigma_{i}(A)^{p}\bigr)^{1/p}$ for
        $1\leq p<\infty$, with $\Phi(x)=\|x\|_{p}$;
  \item the \emph{nuclear norm} $\|A\|_{*}=\sum_{i}\sigma_{i}(A)$, which
        is the Schatten $1$-norm, with $\Phi(x)=\|x\|_{1}$;
  \item the \emph{Ky Fan $k$-norms}
        $\|A\|_{(k)}=\sum_{i=1}^{k}\sigma_{i}(A)$ for $k=1,\dots,n$,
        with $\Phi_{(k)}(x)$ equal to the sum of the $k$ largest among
        $|x_{1}|,\dots,|x_{n}|$.  The cases $k=1$ and $k=n$ are the
        spectral and nuclear norms.
\end{itemize}

Finally, we recall the convexity notion used in the uniqueness
statements of Section~\ref{sec:main}.

\begin{definition}[Strictly convex norm]
    \label{def:strictconvex}
    A norm $N$ on a real vector space is \emph{strictly convex} if
    $N(u)=N(v)$ and $u\neq v$ imply
    $N\bigl(\tfrac{1}{2}(u+v)\bigr)<N(u)$.
\end{definition}

Among the norms above, the Schatten $p$-norms with $1<p<\infty$ are
strictly convex, and the Frobenius norm is the case $p=2$.  The spectral
norm, the nuclear norm and the Ky Fan $k$-norms with $1\leq k<n$ are
not.  Section~\ref{sec:main} shows that strict convexity is sufficient
for uniqueness of the minimiser over the affine family, but not
necessary: the nuclear norm minimiser is also unique.

\subsection{A bordered positive semidefiniteness criterion}
\label{subsec:bordered_psd}

The description of the spectral minimisers in Section~\ref{sec:main}
rests on the following classical criterion, in which $M^{+}$ denotes the
Moore--Penrose pseudo-inverse of $M$.  We include the short proof
because the identity \eqref{eq:schur_min} it produces --- the Schur
complement is the minimum of the associated quadratic form --- is what
gives the conditions of Theorem~\ref{thm:spectral_set} their meaning.

\begin{lemma}[Bordered positive semidefinite matrices]
    \label{lem:bordered_psd}
    Let $M\in\R^{m\times m}$ be symmetric positive semidefinite, let
    $\by\in\R^{m}$ and $a\in\R$, and set
    \[
      N=\begin{pmatrix} a & \by^{\top}\\ \by & M\end{pmatrix}
      \in\Symn[m+1].
    \]
    Then $N\succeq 0$ if and only if $\by\in\operatorname{range}M$ and
    $a\geq\by^{\top}M^{+}\by$.
\end{lemma}

\begin{proof}
    Write $q(t,\bm{x})=at^{2}+2t\,\by^{\top}\bm{x}
    +\bm{x}^{\top}M\bm{x}$ for the quadratic form of $N$.

    Suppose first that $\by\notin\operatorname{range}M$.  Since $M$ is
    symmetric, $\R^{m}=\operatorname{range}M\oplus\ker M$, so
    $\by=\by_{0}+\by_{1}$ with $\by_{0}\in\operatorname{range}M$ and
    $\by_{1}\in\ker M$, $\by_{1}\neq 0$.  Taking $\bm{x}=u\by_{1}$ and
    using $M\by_{1}=0$ and $\by_{0}^{\top}\by_{1}=0$ gives
    $q(1,u\by_{1})=a+2u\|\by_{1}\|^{2}$, which tends to $-\infty$ as
    $u\to-\infty$; hence $N\not\succeq 0$.

    Suppose now $\by\in\operatorname{range}M$ and put
    $\bm{x}_{0}=-M^{+}\by$, so that $M\bm{x}_{0}=-\by$.  Substituting
    $\bm{x}=\bm{x}_{0}+\bm{h}$ and using
    $\bm{x}_{0}^{\top}M\bm{h}=-\by^{\top}\bm{h}$ and
    $\bm{x}_{0}^{\top}M\bm{x}_{0}=-\by^{\top}\bm{x}_{0}
    =\by^{\top}M^{+}\by$, all the terms in $\bm{h}$ of first order
    cancel and
    \[
      q(1,\bm{x})=a-\by^{\top}M^{+}\by+\bm{h}^{\top}M\bm{h}.
    \]
    Since $M\succeq 0$, the minimum over $\bm{x}$ is attained at
    $\bm{h}=0$:
    \begin{equation}\label{eq:schur_min}
      \min_{\bm{x}\in\R^{m}}q(1,\bm{x})=a-\by^{\top}M^{+}\by .
    \end{equation}
    Finally $q(0,\bm{x})=\bm{x}^{\top}M\bm{x}\geq 0$ always, and
    $q(t,\bm{x})=t^{2}q(1,\bm{x}/t)$ for $t\neq 0$, so $N\succeq 0$ if
    and only if $q(1,\cdot)\geq 0$, that is, by \eqref{eq:schur_min}, if
    and only if $a\geq\by^{\top}M^{+}\by$.
\end{proof}

\section{Main result}
\label{sec:main}

Throughout this section $D\in\Symn$ is a fixed dissimilarity matrix,
$\Dtwo=(D_{ij}^{2})$, and $A(\bg)$ is the matrix defined in
\eqref{eq:Ag}; $J=I-\tfrac{1}{n}\ones\ones^{\top}$ is the centring
projection of Definition~\ref{def:centring}, and $\normui{\cdot}$ is an
arbitrary unitarily invariant norm on $\R^{n\times n}$, described by a
symmetric gauge function as in Theorem~\ref{thm:vonneumann}.  All the
notation and terminology used below is that of
Section~\ref{sec:prelim}.

The proof of the main theorem rests on a general compression-reduction
principle (Definition~\ref{def:compclosed} and
Lemma~\ref{lem:compreduction}): whenever a set $\mathcal{A}$ is
compression-closed under an orthogonal projection $P$ and all the
matrices $PAP$, with $A\in\mathcal{A}$, coincide, their common value
simultaneously minimises every unitarily invariant norm over
$\mathcal{A}$.  We apply this principle with
$\mathcal{A}=\mathcal{C}(\Dtwo)$ and $P=J$.
That $\mathcal{C}(\Dtwo)$ is compression-closed under $J$ with singleton
image is specific to the affine family \eqref{eq:Ag}: the centred
compression of $A(\bg)$ is independent of $\bg$
(Lemma~\ref{lem:compression}) and equals $A(\gF)$
(Lemma~\ref{lem:gF}), which belongs to the family.
The contractivity of orthogonal compression for unitarily invariant
norms (Lemma~\ref{lem:proj}) is a classical fact invoked inside the
proof of Lemma~\ref{lem:compreduction}.  Nothing else enters the proof.  In
particular, the strict convexity of the Frobenius objective and the
explicit formula for $\gF$ are not needed for the
\emph{existence} of a universal minimiser; they make that minimiser
explicit and, in the Frobenius case, unique.

\begin{lemma}[Centred compression of $A(\bg)$]
    \label{lem:compression}
    Let $\bg\in\R^n$ be arbitrary. Then
    \begin{equation}\label{eq:centred_compression}
      JA(\bg)J = -\tfrac{1}{2}J\Dtwo J .
    \end{equation}
    In particular, the compression of $A(\bg)$ to $\ones^\perp$ is
    independent of $\bg$.
\end{lemma}

\begin{proof}
    Using $J\ones=0$ and the definition of $A(\bg)$,
    \[
      JA(\bg)J
      =
      \tfrac{1}{2}
      \bigl(
        J\ones\bg^\top J
        +
        J\bg\ones^\top J
        -
        J\Dtwo J
      \bigr)
      =
      -\tfrac{1}{2}J\Dtwo J . \qedhere
    \]
\end{proof}

\begin{lemma}[The Frobenius minimiser]
    \label{lem:gF}
    The problem $\min_{\bg\in\R^{n}}\normF{A(\bg)}$ has a unique
    solution $\gF$, namely
    \begin{equation}\label{eq:gF}
      g^{F}_{k}
      =
      \frac{1}{n}\sum_{i=1}^{n}D_{ik}^{2}
      -
      \frac{1}{2n^{2}}\sum_{i,j=1}^{n}D_{ij}^{2},
      \qquad k=1,\dots,n.
    \end{equation}
    Moreover $\gF$ is the unique parameter with $A(\gF)\ones=0$, and
    \begin{equation}\label{eq:AgF}
      A(\gF)=-\tfrac{1}{2}J\Dtwo J .
    \end{equation}
\end{lemma}

\begin{proof}
    Since $x\mapsto x^{2}$ is strictly increasing on $[0,\infty)$, the
    minimisers of $\normF{A(\bg)}$ are those of
    $f(\bg)=\normF{A(\bg)}^{2}$.  The map $\bg\mapsto A(\bg)$ is affine,
    with
    \begin{equation}\label{eq:affine_perturbation}
      A(\bg+t\bh)
      =
      A(\bg)+\tfrac{t}{2}\bigl(\ones\bh^{\top}+\bh\ones^{\top}\bigr),
      \qquad \bh\in\R^{n},\ t\in\R,
    \end{equation}
    so $f$ is a quadratic polynomial and
    \[
      \frac{d^{2}}{dt^{2}}f(\bg+t\bh)\Big|_{t=0}
      =
      2\bigl\|\tfrac{1}{2}(\ones\bh^{\top}+\bh\ones^{\top})\bigr\|_{F}^{2}.
    \]
    This vanishes only if $h_{i}+h_{j}=0$ for all $i,j$, hence, taking
    $i=j$, only if $\bh=0$.  Thus $f$ is a strictly convex coercive
    quadratic and has a unique minimiser, characterised by its
    stationarity condition.  By \eqref{eq:affine_perturbation} and the
    symmetry of $A(\bg)$,
    \[
      \frac{d}{dt}f(\bg+t\bh)\Big|_{t=0}
      =
      2\,\bh^{\top}A(\bg)\ones ,
    \]
    so that condition is $A(\bg)\ones=0$.

    It remains to solve it.  Writing $s=\ones^{\top}\bg$, the equation
    $A(\bg)\ones=0$ reads $n\bg=\Dtwo\ones-s\ones$; multiplying by
    $\ones^{\top}$ gives $s=\tfrac{1}{2n}\ones^{\top}\Dtwo\ones$, and
    substituting back yields
    \[
      \bg
      =
      \frac{1}{n}\Dtwo\ones
      -
      \frac{1}{2n^{2}}(\ones^{\top}\Dtwo\ones)\ones ,
    \]
    which is \eqref{eq:gF}.  Finally, $A(\gF)$ is symmetric and
    annihilates $\ones$, so $A(\gF)=JA(\gF)J$, and \eqref{eq:AgF}
    follows from Lemma~\ref{lem:compression}.
\end{proof}

\begin{remark}\label{rem:key}
    The structural fact behind Lemma~\ref{lem:compression} is not that
    the spectrum of $A(\bg)$ is obtained from the spectrum on
    $\ones^{\perp}$ by adjoining one $\bg$-dependent eigenvalue.  That
    statement is false in general, because $\ones^{\perp}$ need not be
    invariant under $A(\bg)$ and $\ones$ need not be an eigenvector.
    What is true, and is what the proof uses, is the compression
    identity \eqref{eq:centred_compression}: every matrix of the family
    has the same centred compression, and by Lemma~\ref{lem:gF} the
    parameter $\gF$ is the only one for which the off-centred part
    vanishes, so that the common compression is attained within the
    family.
\end{remark}

The second ingredient is classical.  Let $P$ be an orthogonal
projection, let $Q=I-P$, and let
\[
  \mathcal{P}(M)=PMP+QMQ,
  \qquad
  \mathcal{Q}(M)=PMP-QMQ
\]
denote the \emph{pinching} of $M$ with respect to the decomposition
$\R^{n}=\operatorname{range}P\oplus\operatorname{range}Q$ and its
signed counterpart.  Bhatia's pinching inequality asserts that
$\mathcal{P}$ is a contraction for every unitarily invariant norm, and
the same holds for the compression $M\mapsto PMP$; see
\cite[Ch.~IV]{Bhatia1997} and the survey \cite{Bhatia2000}, where this inequality and several related matrix inequalities are discussed.
This is the only result external to the present setting that the proof
of Theorem~\ref{thm:main} requires.  We include the short argument, in
the form given in \cite{Bhatia2000}, because the identity
\eqref{eq:pinching_convex} appearing in it is used again in
Corollary~\ref{cor:strict_unique}.

\begin{lemma}[Contractivity of orthogonal compression]
    \label{lem:proj}
    Let $M\in\R^{n\times n}$ and let $P\in\R^{n\times n}$ be an
    orthogonal projection.  Then
    \[
      \normui{PMP}\leq\normui{M}
    \]
    for every unitarily invariant norm $\normui{\cdot}$.
\end{lemma}

\begin{proof}
    Set $Q=I-P$ and $R=P-Q$.  Then $Q$ is an orthogonal projection with
    $PQ=QP=0$, and $R^{\top}R=(P-Q)^{2}=P+Q=I$, so $R$ is orthogonal.
    Expanding both $R^{\top}MR$ and $M=(P+Q)M(P+Q)$ gives
    \[
      R^{\top}MR=\mathcal{P}(M)-PMQ-QMP,
      \qquad
      M=\mathcal{P}(M)+PMQ+QMP,
    \]
    so that the pinching is the average of two orthogonally equivalent
    copies of $M$:
    \begin{equation}\label{eq:pinching_convex}
      \mathcal{P}(M)=\tfrac{1}{2}M+\tfrac{1}{2}R^{\top}MR .
    \end{equation}
    By the triangle inequality and unitary invariance,
    \begin{equation}\label{eq:pinching_ineq}
      \normui{\mathcal{P}(M)}
      \leq
      \tfrac{1}{2}\normui{M}+\tfrac{1}{2}\normui{R^{\top}MR}
      =
      \normui{M},
    \end{equation}
    which is the pinching inequality.  Moreover
    \[
      R\,\mathcal{P}(M)
      =
      (P-Q)(PMP+QMQ)
      =
      PMP-QMQ
      =
      \mathcal{Q}(M),
    \]
    so $\normui{\mathcal{Q}(M)}=\normui{\mathcal{P}(M)}$ by unitary
    invariance.  Since
    \begin{equation}\label{eq:compression_convex}
      PMP=\tfrac{1}{2}\mathcal{P}(M)+\tfrac{1}{2}\mathcal{Q}(M),
    \end{equation}
    a further application of the triangle inequality, together with
    \eqref{eq:pinching_ineq}, gives $\normui{PMP}\leq\normui{M}$.
\end{proof}

The two ingredients above motivate the following abstract
reduction principle, which makes the proof of the main theorem one
sentence.

\begin{definition}[Compression-closed set]
    \label{def:compclosed}
    Let $\mathcal{A}\subseteq\Symn$ and let $P$ be an orthogonal
    projection.  The \emph{compression image of $\mathcal{A}$ under $P$} is
    \[
      \CompI{P}{\mathcal{A}}
      \coloneqq
      \{PAP : A\in\mathcal{A}\}.
    \]
    We say that $\mathcal{A}$ is \emph{compression-closed} under $P$ if
    $\CompI{P}{\mathcal{A}}\subseteq\mathcal{A}$.
\end{definition}

\begin{lemma}[Compression-reduction principle]
    \label{lem:compreduction}
    Let $\mathcal{A}\subseteq\Symn$ be compression-closed under an orthogonal
    projection $P$, and suppose $\CompI{P}{\mathcal{A}}=\{A^{*}\}$
    is a singleton.  Then
    \[
      \normui{A^{*}}=\min_{A\in\mathcal{A}}\normui{A}
    \]
    for every unitarily invariant norm $\normui{\cdot}$.
\end{lemma}

\begin{proof}
    For any $A\in\mathcal{A}$, Lemma~\ref{lem:proj} gives
    $\normui{PAP}\leq\normui{A}$.  Since
    $PAP\in\CompI{P}{\mathcal{A}}=\{A^{*}\}$, this yields
    $\normui{A^{*}}\leq\normui{A}$.  Compression-closure ensures
    $A^{*}\in\mathcal{A}$, so the minimum is attained at~$A^{*}$.
\end{proof}

\begin{theorem}[Universal optimality under all unitarily invariant norms]
    \label{thm:main}
    Let $D\in\Symn$ be a dissimilarity matrix in the sense of
    Definition~\ref{def:dissim}, and let $\Dtwo=(D_{ij}^{2})$.
    Define $A(\bg)$ as in \eqref{eq:Ag}, and let $\gF$ be the
    Frobenius minimiser given by \eqref{eq:gF}.  Then $A(\gF)$ minimises
    $\normui{\cdot}$ over $\mathcal{C}(\Dtwo)$ for every unitarily invariant norm
    $\normui{\cdot}$.  In particular, $A(\gF)$ minimises the spectral
    norm, the nuclear norm, the Frobenius norm, the Schatten $p$-norms,
    and the Ky Fan norms over $\mathcal{C}(\Dtwo)$.
\end{theorem}

\begin{proof}
    By Proposition~\ref{prop:param}, which applies because $D_{ii}=0$,
    $\mathcal{C}(\Dtwo)=\{A(\bg):\bg\in\R^{n}\}$.
    By Lemmas~\ref{lem:compression} and~\ref{lem:gF},
    $JA(\bg)J=A(\gF)$ for every $\bg\in\R^{n}$, so
    \[
      \CompI{J}{\mathcal{C}(\Dtwo)}=\{A(\gF)\}.
    \]
    Since $A(\gF)\in\mathcal{C}(\Dtwo)$ (Proposition~\ref{prop:param}),
    the compatible set is compression-closed under $J$ with singleton image.
    Lemma~\ref{lem:compreduction} gives
    \[
      \normui{A(\gF)}=\min_{A\in\mathcal{C}(\Dtwo)}\normui{A}
    \]
    for every unitarily invariant norm.
\end{proof}

The preceding theorem identifies a common minimiser.  We next record two
uniqueness consequences.  The first one answers explicitly what happens
for the nuclear norm, whose gauge function is not strictly convex.  The
inequality $\|PMP\|_{*}\leq\|M\|_{*}$ is of course the nuclear-norm case
of Lemma~\ref{lem:proj}; the point of the next lemma is that proving it
again along a different route, through the positive and negative parts
of $M$, makes the equality case explicit.

\begin{lemma}[Equality in the nuclear compression bound]
    \label{lem:nuclear_unique}
    Let $M\in\R^{n\times n}$ be symmetric, let $P$ be an orthogonal
    projection, and set $Q=I-P$.  If
    \[
      \|PMP\|_*=\|M\|_*,
    \]
    then $MQ=QM=0$; equivalently, $M=PMP$.
\end{lemma}

\begin{proof}
    Since $M$ is symmetric, it decomposes into positive and negative
    parts, $M=M_{+}-M_{-}$, where $M_{\pm}\succeq 0$ are obtained from
    the spectral decomposition of $M$ by retaining respectively its
    positive and its negative eigenvalues.  The singular values of $M$
    are the absolute values of its eigenvalues, so
    \begin{equation}\label{eq:nuc_parts}
      \|M\|_{*}=\tr(M_{+})+\tr(M_{-}).
    \end{equation}

    We first bound $\|PMP\|_{*}$.  Each of $PM_{+}P$ and $PM_{-}P$ is
    positive semidefinite, being of the form
    $(M_{\pm}^{1/2}P)^{\top}(M_{\pm}^{1/2}P)$, and the nuclear norm of a
    positive semidefinite matrix is its trace.  Applying the triangle
    inequality to $PMP=PM_{+}P-PM_{-}P$ therefore gives
    \begin{equation}\label{eq:nuc_triangle}
      \|PMP\|_{*}
      \leq
      \|PM_{+}P\|_{*}+\|PM_{-}P\|_{*}
      =
      \tr(PM_{+}P)+\tr(PM_{-}P).
    \end{equation}
    Since $P^{2}=P$, the cyclicity of the trace gives
    $\tr(PM_{\pm}P)=\tr(M_{\pm}P)$, and $P+Q=I$ then yields the identity
    \begin{equation}\label{eq:nuc_defect}
      \tr(PM_{\pm}P)=\tr(M_{\pm})-\tr(QM_{\pm}Q).
    \end{equation}
    Substituting \eqref{eq:nuc_defect} into \eqref{eq:nuc_triangle} and
    using \eqref{eq:nuc_parts},
    \begin{equation}\label{eq:nuc_gap}
      \|PMP\|_{*}
      \leq
      \|M\|_{*}-\bigl[\tr(QM_{+}Q)+\tr(QM_{-}Q)\bigr].
    \end{equation}
    The two bracketed terms are non-negative because
    $QM_{\pm}Q\succeq 0$, so \eqref{eq:nuc_gap} contains the inequality
    $\|PMP\|_{*}\leq\|M\|_{*}$ and, in addition, identifies its defect.

    Assume now that equality holds.  Then the bracket in
    \eqref{eq:nuc_gap} vanishes and, both of its terms being
    non-negative,
    \[
      \tr(QM_{+}Q)=\tr(QM_{-}Q)=0 .
    \]
    A positive semidefinite matrix of zero trace is zero, so
    $QM_{+}Q=QM_{-}Q=0$.  Writing $M_{\pm}=C_{\pm}^{\top}C_{\pm}$ with
    $C_{\pm}=M_{\pm}^{1/2}$, this says that
    \[
      \|C_{\pm}Q\|_{F}^{2}
      =
      \tr\bigl(QC_{\pm}^{\top}C_{\pm}Q\bigr)
      =
      \tr(QM_{\pm}Q)
      =
      0,
    \]
    whence $C_{\pm}Q=0$ and therefore
    $M_{\pm}Q=C_{\pm}^{\top}C_{\pm}Q=0$.  Consequently
    $MQ=(M_{+}-M_{-})Q=0$ and, $M$ being symmetric,
    $QM=(MQ)^{\top}=0$.  Finally $M=(P+Q)M(P+Q)=PMP$.
\end{proof}

\begin{corollary}[Uniqueness for Frobenius and nuclear norms]
    \label{cor:FrobNucUnique}
    Under the hypotheses of Theorem~\ref{thm:main}, $\gF$ is the unique
    minimiser of both
    \[
      \min_{\bg\in\R^n}\normF{A(\bg)}
      \qquad\text{and}\qquad
      \min_{\bg\in\R^n}\|A(\bg)\|_* .
    \]
\end{corollary}

\begin{proof}
    The Frobenius assertion is Lemma~\ref{lem:gF}.  For the
    nuclear norm, let $\bg$ be a minimiser.  By Theorem~\ref{thm:main}
    and Lemma~\ref{lem:compression},
    \[
      \|A(\bg)\|_*=\|A(\gF)\|_*=\|JA(\bg)J\|_* .
    \]
    Lemma~\ref{lem:nuclear_unique}, applied with $M=A(\bg)$ and $P=J$,
    gives $A(\bg)=JA(\bg)J$, and the right-hand side is $A(\gF)$ by
    Lemmas~\ref{lem:compression} and~\ref{lem:gF}; hence
    \[
      A(\bg)=A(\gF).
    \]
    Hence
    $\ones(\bg-\gF)^\top+(\bg-\gF)\ones^\top=0$, which implies
    $\bg=\gF$ by taking diagonal entries.  Thus the nuclear norm
    minimiser is unique.
\end{proof}

\begin{corollary}[Uniqueness for strictly convex unitarily invariant norms]
    \label{cor:strict_unique}
    If the unitarily invariant norm in Theorem~\ref{thm:main} is
    strictly convex as a norm on $\R^{n\times n}$, then $\gF$ is its
    unique minimiser over the affine family \eqref{eq:Ag}.
\end{corollary}

\begin{proof}
    Let $\bg$ be a minimiser and set $M=A(\bg)$, $P=J$ and $Q=I-P$.
    By Theorem~\ref{thm:main} and Lemma~\ref{lem:compression},
    \[
      \normui{M}=\normui{A(\gF)}=\normui{JMJ}=\normui{PMP},
    \]
    so both inequalities used in the proof of Lemma~\ref{lem:proj} are
    equalities for this $M$.

    Equality in \eqref{eq:pinching_ineq} means that
    $\normui{\mathcal{P}(M)}=\normui{M}=\normui{R^{\top}MR}$ while, by
    \eqref{eq:pinching_convex}, $\mathcal{P}(M)$ is the midpoint of $M$
    and $R^{\top}MR$.  Strict convexity of the norm forces
    $M=R^{\top}MR=(P-Q)M(P-Q)$, whence $PMQ=QMP=0$.

    Equality in the last step, applied to \eqref{eq:compression_convex},
    gives $\normui{\mathcal{P}(M)}=\normui{\mathcal{Q}(M)}=\normui{PMP}$
    with $PMP$ the midpoint of $\mathcal{P}(M)$ and $\mathcal{Q}(M)$;
    strict convexity forces $\mathcal{P}(M)=\mathcal{Q}(M)$, that is
    $QMQ=0$.

    Therefore $M=PMP+PMQ+QMP+QMQ=PMP=JA(\bg)J=A(\gF)$, so
    $\ones(\bg-\gF)^{\top}+(\bg-\gF)\ones^{\top}=0$, and taking
    diagonal entries gives $\bg=\gF$.
\end{proof}

    \begin{remark}[Variational characterisation of the double-centred matrix]
        \label{rem:CMDS}
        Identity \eqref{eq:AgF} shows that $A(\gF)$ is precisely the
        double-centred matrix $-\tfrac{1}{2}J\Dtwo J$ that lies at the heart
        of classical multidimensional scaling (CMDS), of principal coordinate
        analysis (PCoA) in numerical ecology and related applications, and of
        the theory of Euclidean distance
        matrices~\cite{Torgerson1952,Gower1966,Schoenberg1938,%
        Krislock2012,Dokmanic2015}.
        The present result provides a purely variational characterisation of
        this canonical object: it is the representative of the affine family
        \eqref{eq:Ag} selected by the Frobenius norm, and it is also a
        minimiser for every unitarily invariant norm.  In particular, the
        classical double-centred matrix can be recovered without imposing
        Euclidean realisability and without privileging the Frobenius norm
        over other unitarily invariant notions of matrix size.
        \end{remark}

    \begin{remark}[Norm-dependent uniqueness structure]
    \label{rem:uniqueness}
    The examples of Section~\ref{sec:example} illustrate a dichotomy that
    reflects the geometry of each gauge function:
    \begin{itemize}
      \item \emph{Strictly convex norms} (Schatten-$p$, $1<p<\infty$,
            including the Frobenius norm ($p=2$)):
            $\gF$ is the unique minimiser by
            Corollary~\ref{cor:strict_unique}.  The nuclear norm is not
            strictly convex, but Corollary~\ref{cor:FrobNucUnique} shows
            that its minimiser is also unique in the present affine family.
      \item \emph{Other non-strictly convex norms} (for instance the
            spectral norm and Ky Fan $k$-norms for $1\leq k<n$): the
            minimiser need not be unique.  Along the slice
            $\bg(t)=\gF+\tfrac{t}{n}\ones$,
            the exact formula $\normtwo{A(\bg(t))}=\max\{\rho^*,|t|\}$
            shows that every $\bg(t)$ with $|t|\leq\rho^*$ is a spectral
            minimiser, while $\gF$ is the only one that also minimises the
            Frobenius and nuclear norms.  For the spectral norm the whole
            minimising set is determined in
            Theorem~\ref{thm:spectral_set}; this slice is the part of it
            obtained by moving in the direction $\ones$.
    \end{itemize}
    This behaviour is consistent with the general picture in matrix 
    nearness problems~\cite{Higham1988}: for the nearest 
    positive-semidefinite matrix problem, Higham shows that the 
    Frobenius minimiser is unique and explicitly computable, whereas 
    spectral-norm minimisers need not be unique and are characterised 
    indirectly through Halmos' scalar condition.  
    Thus the spectral non-uniqueness exhibited by the slab 
    $|t|\le \rho^*$ is in line with the phenomena occurring in 
    Higham's spectral-norm setting.
    \end{remark}

Theorem~\ref{thm:main} says that all unitarily invariant norms are
minimised at the same parameter, but not that they take the same value
there, and in general they do not.  The next corollary singles out the
degenerate case in which the Frobenius and spectral norms agree at
the optimum; equality of these two norms means that $A(\gF)$ has at
most one nonzero singular value, so all Schatten norms coincide.
The case is worth isolating for two reasons: it is exactly the situation in
which the optimal value, and not only the optimal parameter, is the
same for all Schatten norms, and, when the data are Euclidean, it
has a direct geometric meaning --- the configuration is collinear
(Corollary~\ref{cor:rank1_euclidean}).

\begin{corollary}[Norm equality and rank condition]
    \label{cor:rank}
    Let $\gF$ be as in Theorem~\ref{thm:main}.  The following assertions
    are equivalent:
    \begin{enumerate}
      \item[(i)]   $\displaystyle\min_{\bg}\normF{A(\bg)}
                   = \min_{\bg}\normtwo{A(\bg)}$;
      \item[(ii)]  $\normF{A(\gF)} = \normtwo{A(\gF)}$;
      \item[(iii)] $\rank(A(\gF))\leq 1$;
      \item[(iv)]  $\rank(J\Dtwo J)\leq 1$.
    \end{enumerate}
    \end{corollary}
    
    \begin{proof}
    By Lemma~\ref{lem:gF}, $\gF$ is the unique Frobenius
    minimiser.  By Theorem~\ref{thm:main}, $\gF$ is also a spectral-norm
    minimiser.  Hence
    \[
      \min_{\bg}\normF{A(\bg)}
      =
      \normF{A(\gF)},
      \qquad
      \min_{\bg}\normtwo{A(\bg)}
      =
      \normtwo{A(\gF)}.
    \]
    This proves the equivalence of (i) and (ii).
    
    For any real symmetric matrix $M$,
    \[
      \normF{M}^{2}=\sum_i\lambda_i(M)^2,
      \qquad
      \normtwo{M}=\max_i|\lambda_i(M)|.
    \]
    Thus $\normF{M}=\normtwo{M}$ if and only if $M$ has at most one
    nonzero eigenvalue, i.e. if and only if $\rank(M)\leq 1$.  Applying
    this to $M=A(\gF)$ proves the equivalence of (ii) and (iii).
    
    Finally, by Lemma~\ref{lem:gF},
    \[
      A(\gF)=-\tfrac{1}{2}J\Dtwo J,
    \]
    so $\rank(A(\gF))=\rank(J\Dtwo J)$.  This proves the equivalence of
    (iii) and (iv).
    \end{proof}

\begin{proposition}[Algebraic characterisation of the rank condition]
\label{prop:rank1geom}
Let $D\in\R^{n\times n}$ be a symmetric dissimilarity matrix.
The following assertions are equivalent:
\begin{enumerate}
  \item[(i)]   $\rank(J\Dtwo J)\leq 1$.
  \item[(ii)]  There exists a vector $\bw\in\ones^{\perp}$,
               $\|\bw\|=1$, and a scalar $\lambda\in\R$ such that
               $J\Dtwo J = \lambda\, \bw\bw^{\top}$.
               Equivalently, $-\tfrac{1}{2}J\Dtwo J$
               has at most one nonzero eigenvalue.
  \item[(iii)] There exists $\bw\in\ones^{\perp}$, $\|\bw\|=1$, and
               $\lambda\in\R$ such that the double-centred matrix
               satisfies
               \begin{equation}\label{eq:rank1form}
                 \bigl(J\Dtwo J\bigr)_{ij}
                 \;=\;
                 D^2_{ij} - \bar{d}_i - \bar{d}_j + \bar{d}
                 \;=\;
                 \lambda\, \bw_i \bw_j,
               \end{equation}
               for all $i,j$, where $\bar{d}_i = \tfrac{1}{n}\sum_k D^2_{ik}$
               and $\bar{d} = \tfrac{1}{n^2}\sum_{k,l}D^2_{kl}$.
               That is, the double-centred squared dissimilarities
               factor as a rank-at-most-one outer product $\lambda\,\bw\bw^\top$
               with $\bw\in\ones^{\perp}$.
\end{enumerate}
In any of these cases, $\normF{A(\gF)}=\normtwo{A(\gF)}=
\tfrac{1}{2}|\lambda|$, and all Schatten norms of
$A(\gF)$ coincide:
\[
  \|A(\gF)\|_{S_p}
  \;=\;
  \normtwo{A(\gF)}
  \;=\;
  \normF{A(\gF)}
  \quad\text{for all }1\leq p\leq\infty.
\]
\end{proposition}

\begin{proof}
\textbf{(i)\,$\Leftrightarrow$\,(ii).}
The matrix $J\Dtwo J$ is real and symmetric, so
$\rank(J\Dtwo J)\leq 1$ if and only if $J\Dtwo J=\lambda\,\bw\bw^{\top}$
for some $\lambda\in\R$ and unit vector $\bw$.
Since $J\ones=0$, every eigenvector of $J\Dtwo J$ for a nonzero
eigenvalue lies in $\ones^{\perp}$.  If $\lambda\neq 0$, this forces
$\bw\in\ones^{\perp}$.  If $\lambda=0$, then $J\Dtwo J=0$
and $\bw$ may be any unit vector in $\ones^{\perp}$, which is
non-empty since $n\geq 2$.

\textbf{(i)\,$\Leftrightarrow$\,(iii).}
Write $B = J\Dtwo J$ with entries
$B_{ij} = D^2_{ij} - \bar{d}_i - \bar{d}_j + \bar{d}$.
Since $B$ is real and symmetric, $\rank(B)\leq 1$ if and only if
$B = \lambda\,\bw\bw^{\top}$ for some $\lambda\in\R$ and unit $\bw$,
which is precisely \eqref{eq:rank1form}.
If $\lambda\neq 0$, the constraint $\bw\in\ones^{\perp}$
follows from $B\ones=0$ (since $J\ones=0$); if $\lambda=0$
then $B=0$ and $\bw$ may be any unit vector in $\ones^{\perp}$,
which is non-empty since $n\geq 2$.

\textbf{Norm equalities.}
When $\rank(A(\gF))\leq 1$, write
$A(\gF)=-\tfrac{1}{2}\lambda\,\bw\bw^{\top}$.
This rank-at-most-one matrix has singular values
$\bigl(\tfrac{1}{2}|\lambda|,\,0,\,\dots,\,0\bigr)$, so
\[
  \|A(\gF)\|_{S_p}=\tfrac{1}{2}|\lambda|  \quad\text{for all }1\leq p\leq\infty,
\]
covering $p=\infty$ (the spectral norm); in particular,
$\normtwo{A(\gF)}=\normF{A(\gF)}=\tfrac{1}{2}|\lambda|$.
\end{proof}

\begin{corollary}[Euclidean interpretation of the rank condition]
    \label{cor:rank1_euclidean}
    Suppose that $D$ is the distance matrix of a configuration
    $\bx_1,\dots,\bx_n\in\R^d$.  Then $\rank(J\Dtwo J)\leq 1$ if and
    only if the configuration has affine dimension at most one.
    If $d\geq 1$, this is equivalent to all $n$ points lying on a common
    affine line.
\end{corollary}

\begin{proof}
    When $D$ is Euclidean, $-\tfrac{1}{2}J\Dtwo J = X^{\top}X$, where
    $X\in\R^{d\times n}$ is the centred configuration matrix whose
    columns are $\bx_i-\bar{\bx}$.  Hence
    \[
      \rank(-\tfrac{1}{2}J\Dtwo J)=\rank(X^{\top}X)=\rank(X).
    \]
    This rank is the affine dimension of the point cloud
    $\{\bx_1,\dots,\bx_n\}$~\cite{YoungHouseholder1938,Dokmanic2015}.
    The affine dimension is at most one exactly when all points are
    collinear.
\end{proof}

\begin{remark}[Geometric meaning of the rank condition]
    \label{rem:nodepotential}
    The characterisation \eqref{eq:rank1form} says that the
    double-centred squared dissimilarities factor as a single outer
    product $\lambda\,\bw\bw^\top$ with $\bw\in\ones^\perp$.
    This is the most degenerate non-trivial structure for the centred
    matrix: after row-and-column centring, all residual information is
    contained in a single direction.
    
    In the Euclidean rank-one case, the configuration is collinear.  After
    choosing an affine coordinate on the line, write
    $\bx=(x_1,\dots,x_n)^\top$ for the scalar coordinates of the points.
    Then
    \[
      -\tfrac{1}{2}J\Dtwo J = (J\bx)(J\bx)^\top .
    \]
    Consequently, if $J\bx\neq 0$, the factorisation
    \[
      J\Dtwo J = \lambda\,\bw\bw^\top
    \]
    is obtained with
    \[
      \bw=\frac{J\bx}{\|J\bx\|},
      \qquad
      \lambda=-2\|J\bx\|^2 .
    \]
    The optimal matrix is therefore
    \[
      A(\gF)=-\tfrac{1}{2}\lambda\,\bw\bw^\top=(J\bx)(J\bx)^\top,
    \]
    which is positive semidefinite and has rank at most one, with unique
    nonzero eigenvalue $-\tfrac{1}{2}\lambda=\|J\bx\|^2$.  In this
    case all Schatten norms of $A(\gF)$ coincide and are equal to
    \[
      \tfrac{1}{2}|\lambda|=\|J\bx\|^2,
    \]
    the total centred sum of squares of the one-dimensional configuration.
    \end{remark}

Theorem~\ref{thm:main} identifies a common minimiser, and
Corollaries~\ref{cor:FrobNucUnique} and~\ref{cor:strict_unique} show that
it is the only one for the nuclear norm and for strictly convex
unitarily invariant norms.  For the spectral norm it is not, and we now
determine the whole set of minimisers.  The description rests on a
change of coordinates that makes the affine family transparent.

\begin{lemma}[Bordered form of the affine family]
    \label{lem:bordered}
    Let $\be=\tfrac{1}{\sqrt n}\ones$, let $Q\in\R^{n\times(n-1)}$ have
    orthonormal columns spanning $\ones^{\perp}$, so that
    $[\,\be\;Q\,]$ is orthogonal, and put
    $B_{0}=Q^{\top}A(\gF)Q\in\R^{(n-1)\times(n-1)}$.  For $\bg\in\R^{n}$
    write $\bg-\gF=c\ones+\bv$ with $c\in\R$ and $\bv\in\ones^{\perp}$,
    and set
    \begin{equation}\label{eq:border_coords}
      s=cn,
      \qquad
      \by=\tfrac{\sqrt n}{2}\,Q^{\top}\bv .
    \end{equation}
    Then
    \begin{equation}\label{eq:bordered}
      [\,\be\;Q\,]^{\top}A(\bg)\,[\,\be\;Q\,]
      =
      \begin{pmatrix}
        s & \by^{\top}\\[2pt]
        \by & B_{0}
      \end{pmatrix},
    \end{equation}
    and $\bg\mapsto(s,\by)$ is an affine bijection of $\R^{n}$ onto
    $\R\times\R^{n-1}$.
\end{lemma}

\begin{proof}
    By \eqref{eq:Ag},
    \[
      A(\bg)=A(\gF)+c\,\ones\ones^{\top}
      +\tfrac{1}{2}\bigl(\ones\bv^{\top}+\bv\ones^{\top}\bigr).
    \]
    Since $A(\gF)\ones=0$ by Lemma~\ref{lem:gF}, $Q^{\top}\ones=0$ and
    $\ones^{\top}\bv=0$, evaluating the three blocks gives
    $\be^{\top}A(\bg)\be=c\,n$,
    $Q^{\top}A(\bg)\be=\tfrac{\sqrt n}{2}Q^{\top}\bv$ and
    $Q^{\top}A(\bg)Q=Q^{\top}A(\gF)Q=B_{0}$, which is
    \eqref{eq:bordered}.  The map $\bg\mapsto(c,\bv)$ is an affine
    bijection of $\R^{n}$ onto $\R\times\ones^{\perp}$, and
    \eqref{eq:border_coords} is a linear bijection of
    $\R\times\ones^{\perp}$ onto $\R\times\R^{n-1}$.
\end{proof}

\begin{remark}[The family as a bordered completion problem]
    \label{rem:bordered}
    Lemma~\ref{lem:bordered} says that, in coordinates adapted to
    $\R^{n}=\operatorname{span}(\ones)\oplus\ones^{\perp}$, the affine
    family \eqref{eq:Ag} is \emph{exactly} the set of symmetric matrices
    whose compression to $\ones^{\perp}$ equals $B_{0}$: the data fix
    the trailing block and leave the border $(s,\by)$ free.  This is the
    coordinate form of Lemma~\ref{lem:compression}, and it recasts the
    problem as one of choosing a border so as to minimise a matrix norm.
    Theorem~\ref{thm:main} says that the zero border is optimal for every
    unitarily invariant norm simultaneously.  This is not a formal
    consequence of the change of coordinates: the spectrum of $B_{0}$
    interlaces that of the bordered matrix rather than being simply
    shifted by it, and the assertion that no unitarily invariant norm can
    decrease is exactly the content of Lemma~\ref{lem:proj}.
\end{remark}

\begin{theorem}[The set of spectral minimisers]
    \label{thm:spectral_set}
    Let $\rho^{*}=\normtwo{A(\gF)}$ and set
    \begin{equation}\label{eq:extremal_eigenspace}
      E=\ker\bigl(A(\gF)-\rho^{*}I\bigr)+\ker\bigl(A(\gF)+\rho^{*}I\bigr),
    \end{equation}
    the span of the eigenvectors of $A(\gF)$ for the eigenvalues of
    largest modulus.  For $\bg\in\R^{n}$ write $\bg-\gF=c\ones+\bv$ with
    $\bv\in\ones^{\perp}$, and define
    \begin{equation}\label{eq:admissible_interval}
      \alpha(\bv)=\tfrac{n}{4}\,\bv^{\top}(\rho^{*}J-A(\gF))^{+}\bv,
      \qquad
      \beta(\bv)=\tfrac{n}{4}\,\bv^{\top}(\rho^{*}J+A(\gF))^{+}\bv.
    \end{equation}
    Then $\bg$ is a spectral minimiser, i.e.\
    \[
      \normtwo{A(\bg)}=\rho^{*}=\min_{\bg'\in\R^{n}}\normtwo{A(\bg')},
    \]
    if and only if
    \begin{equation}\label{eq:spectral_conditions}
      \bv\perp E
      \qquad\text{and}\qquad
      \beta(\bv)-\rho^{*}\leq nc\leq\rho^{*}-\alpha(\bv).
    \end{equation}
    The scalar inequalities in \eqref{eq:spectral_conditions} define a
    nonempty interval for $nc$ if and only if
    $\alpha(\bv)+\beta(\bv)\leq 2\rho^{*}$.
    When nonempty, this interval has length
    $2\rho^{*}-[\alpha(\bv)+\beta(\bv)]$; for $\bv=0$ it is
    $[-\rho^{*},\rho^{*}]$.  Consequently the set of
    spectral minimisers is compact and convex; it reduces to $\{\gF\}$
    exactly when $\rho^{*}=0$; and when $\rho^{*}>0$ it is contained in
    the affine subspace
    \[
      \gF+\operatorname{span}(\ones)+(\ones^{\perp}\cap E^{\perp}),
    \]
    which has dimension $n-\dim E\leq n-1$.
\end{theorem}

\begin{proof}
    By Lemma~\ref{lem:bordered} and the orthogonal invariance of the
    spectral norm, $\normtwo{A(\bg)}$ is the spectral norm of the
    bordered matrix
    \[
      S=\begin{pmatrix} s & \by^{\top}\\ \by & B_{0}\end{pmatrix}
    \]
    of \eqref{eq:bordered}, and by Theorem~\ref{thm:main} the minimum
    value over the family is $\rho^{*}$.  Hence $\bg$ is a spectral
    minimiser if and only if $\normtwo{S}\leq\rho^{*}$, which for a
    symmetric matrix means
    \[
      \rho^{*}I-S\succeq 0
      \qquad\text{and}\qquad
      \rho^{*}I+S\succeq 0 .
    \]
    Since $A(\gF)\ones=0$, the bordered form applied to $\bg=\gF$ gives
    border $(0,0)$ and $\rho^{*}=\normtwo{B_{0}}$; hence
    $M_{+}:=\rho^{*}I-B_{0}\succeq 0$ and
    $M_{-}:=\rho^{*}I+B_{0}\succeq 0$.  Set
    $\widehat\alpha(\by):=\by^{\top}M_{+}^{+}\by$ and
    $\widehat\beta(\by):=\by^{\top}M_{-}^{+}\by$.  Explicitly,
    \[
      \rho^{*}I-S=
      \begin{pmatrix}\rho^{*}-s & -\by^{\top}\\ -\by & M_{+}\end{pmatrix},
      \qquad
      \rho^{*}I+S=
      \begin{pmatrix}\rho^{*}+s & \by^{\top}\\ \by & M_{-}\end{pmatrix},
    \]
    so Lemma~\ref{lem:bordered_psd}, applied to each of them, says that
    $\bg$ is a spectral minimiser if and only if
    \begin{equation}\label{eq:spectral_two_conditions}
      \by\in\operatorname{range}M_{+}\cap\operatorname{range}M_{-},
      \qquad
      \rho^{*}-s\geq\widehat\alpha(\by),
      \qquad
      \rho^{*}+s\geq\widehat\beta(\by).
    \end{equation}
    By \eqref{eq:schur_min} the two inequalities say precisely that the
    quadratic forms of $\rho^{*}I\pm S$, minimised over the last $n-1$
    coordinates with the first fixed to~$1$, are non-negative;
    $\widehat\alpha(\by)$ and $\widehat\beta(\by)$ measure how much the border lowers
    those minima.

    We now show that $\widehat\alpha(\by)=\alpha(\bv)$ and
    $\widehat\beta(\by)=\beta(\bv)$ as defined in
    \eqref{eq:admissible_interval}.  To this end, we claim that the identity
    \begin{equation}\label{eq:pseudoinv_Q}
      (Q^{\top}AQ)^{+}=Q^{\top}(QQ^{\top}AQQ^{\top})^{+}Q ,
    \end{equation}
    is valid for every $A\in\Symn$ and every $Q\in\R^{n\times r}$ with
    $Q^{\top}Q=I_{r}$.  To prove the claim, complete $Q$ to a full orthogonal matrix
    $W=[Q\;Q_{\perp}]\in\R^{n\times n}$.  In the basis $W$,
    $QQ^{\top}AQQ^{\top}=W\operatorname{diag}(Q^{\top}AQ,\,0)\,W^{\top}$,
    so 
    \[(QQ^{\top}AQQ^{\top})^{+}=W\operatorname{diag}((Q^{\top}AQ)^{+},\,0)\,W^{\top},\]
    and therefore
    \[
      Q^{\top}(QQ^{\top}AQQ^{\top})^{+}Q
      =\bigl[I_{r}\;0\bigr]\operatorname{diag}\!\bigl((Q^{\top}AQ)^{+},0\bigr)
      \bigl[I_{r}\;0\bigr]^{\top}=(Q^{\top}AQ)^{+}.
    \]
    Applying \eqref{eq:pseudoinv_Q} first with
    $A=\rho^{*}I-A(\gF)$ and then with $A=\rho^{*}I+A(\gF)$,
    and using $QQ^{\top}=J$ together with $JA(\gF)J=A(\gF)$
    (Lemma~\ref{lem:gF}), the relations $Q\by=\tfrac{\sqrt{n}}{2}\bv$
    and $J(\rho^{*}I\mp A(\gF))J=\rho^{*}J\mp A(\gF)$ give
    \[
      \widehat\alpha(\by)=\tfrac{n}{4}\,\bv^{\top}(\rho^{*}J-A(\gF))^{+}\bv,
      \qquad
      \widehat\beta(\by)=\tfrac{n}{4}\,\bv^{\top}(\rho^{*}J+A(\gF))^{+}\bv,
    \]
    which are exactly $\alpha(\bv)$ and $\beta(\bv)$ from
    \eqref{eq:admissible_interval}, independent of the choice of $Q$.
    Hence the two scalar inequalities in \eqref{eq:spectral_two_conditions}
    are equivalent to
    $\beta(\bv)-\rho^{*}\leq nc\leq\rho^{*}-\alpha(\bv)$ (using $s=nc$
    from \eqref{eq:border_coords}), which is the second condition in
    \eqref{eq:spectral_conditions}.

    It remains to identify the range condition in
    \eqref{eq:spectral_two_conditions} with $\bv\perp E$.  Since
    $M_{\pm}$ is symmetric,
    $\operatorname{range}M_{\pm}=(\ker M_{\pm})^{\perp}$, and $\ker
    M_{+}$, $\ker M_{-}$ are the eigenspaces of $B_{0}$ for the
    eigenvalues $\rho^{*}$ and $-\rho^{*}$.  By \eqref{eq:bordered}
    applied to $\bg=\gF$ we have $B_{0}=Q^{\top}A(\gF)Q$ with
    $A(\gF)\ones=0$, so $\bv\mapsto Q^{\top}\bv$ is an isometry of
    $\ones^{\perp}$ onto $\R^{n-1}$ carrying the eigenspaces of $A(\gF)$
    inside $\ones^{\perp}$ onto those of $B_{0}$.  As
    $\by=\tfrac{\sqrt n}{2}Q^{\top}\bv$, the condition
    $\by\perp\ker M_{+}\cup\ker M_{-}$ is therefore $\bv\perp E$, in
    both cases $\rho^{*}>0$ and $\rho^{*}=0$.

    Finally, by Proposition~\ref{prop:param} the map $\bg\mapsto A(\bg)$
    is an affine bijection from $\R^{n}$ onto $\mathcal{C}(\Dtwo)$.  The
    set of spectral minimisers in $\R^{n}$ is therefore the inverse image
    of
    \[
      \mathcal{C}(\Dtwo)\cap\{A\in\Symn:\normtwo{A}\leq\rho^{*}\},
    \]
    which is the intersection of an affine subspace with a closed convex
    set, hence convex and closed.  If $\rho^{*}=0$ then $E=\R^{n}$ forces $\bv=0$,
    and $\alpha(0)=\beta(0)=0$ so the interval in
    \eqref{eq:spectral_conditions} reduces to $\{0\}$ and forces $c=0$,
    giving $\{\gF\}$; conversely if $\rho^{*}>0$ then $\bv=0$,
    $|nc|\leq\rho^{*}$ already gives infinitely many minimisers.  For boundedness,
    $|a_{ii}(\bg)|\leq\normtwo{A(\bg)}=\rho^{*}$
    for all $i$, and \eqref{eq:Ag} gives $a_{ii}(\bg)=g_{i}-\tfrac{1}{2}D_{ii}^{2}$,
    so every minimiser satisfies $|g_{i}-\tfrac{1}{2}D_{ii}^{2}|\leq\rho^{*}$
    for all $i$; the minimising set is therefore bounded in $\R^{n}$.
    The stated inclusion is the condition $\bv\perp E$ together with the
    freedom in $c$, and $\dim E\geq 1$ because $\rho^{*}$ is attained as
    the modulus of an eigenvalue of $A(\gF)$.
\end{proof}

\begin{corollary}[Uniqueness and non-uniqueness for the spectral norm]
    \label{cor:spectral_unique}
    If $\rho^{*}=0$, then $\gF$ is the unique spectral minimiser.  If
    $\rho^{*}>0$, the spectral minimiser is never unique: taking
    $\bv=0$ in Theorem~\ref{thm:spectral_set} gives $\alpha(0)=\beta(0)=0$,
    so every $\bg=\gF+c\ones$ with $|nc|\leq\rho^{*}$ is a minimiser.
\end{corollary}

\begin{proof}
    Both assertions are contained in Theorem~\ref{thm:spectral_set}.  If
    $\rho^{*}=0$ the minimising set is $\{\gF\}$.  If $\rho^{*}>0$,
    taking $\bv=0$ gives $\alpha(0)=\beta(0)=0$ so the interval in
    \eqref{eq:spectral_conditions} is $[-\rho^{*},\rho^{*}]$, and
    $\bv=0$ satisfies $\bv\perp E$ trivially; the minimisers with
    $\bv=0$ are exactly $\bg=\gF+c\ones$ with $|nc|\leq\rho^{*}$.
\end{proof}

\begin{remark}[Reading the two conditions]
    \label{rem:spectral_reading}
    The condition $\bv\perp E$ constrains the \emph{direction} of the
    displacement and \eqref{eq:spectral_conditions} its \emph{magnitude}.
    Moving away from $\gF$ along a direction orthogonal to the extremal
    eigenvectors of $A(\gF)$ leaves the spectral norm unchanged provided
    one does not move too far; the first condition says that no other
    direction has this property, however small the displacement, since any
    component of $\bv$ along an eigenvector for $\pm\rho^{*}$ makes the
    bordered matrix exceed $\rho^{*}$ in norm.  The second condition says
    how far is too far: the admissible interval for $nc$ is
    $[-\rho^{*},\rho^{*}]$ when $\bv=0$, and a nonzero $\bv$ shortens
    it by $\alpha(\bv)$ on the right and $\beta(\bv)$ on the left,
    closing it altogether once $\alpha(\bv)+\beta(\bv)>2\rho^{*}$.
    The slice $\bg(t)=\gF+\tfrac{t}{n}\ones$ used in
    Section~\ref{sec:example} is the case $\bv=0$.

    Two consequences are worth recording.  First, the non-uniqueness is
    thin: by Theorem~\ref{thm:spectral_set} the minimising set lies in an
    affine subspace of dimension $n-\dim E$, so it has empty interior in
    $\R^{n}$.  Second, Corollary~\ref{cor:strict_unique} is the opposite
    extreme: for a strictly convex unitarily invariant norm the
    admissible set collapses to the single point $\gF$.
\end{remark}

\section{Numerical examples}
\label{sec:example}

We illustrate the main results with three examples of increasing
complexity.  All numerical values have been verified by exact
rational arithmetic or independent floating-point computation.

\begin{example}[$n=4$, non-Euclidean dissimilarity matrix]
    \label{ex:noneuc4}
    Let
    \[
      \Dtwo =
      \begin{pmatrix}
         0 &  9 &  4 & 16 \\
         9 &  0 &  1 &  9 \\
         4 &  1 &  0 &  4 \\
        16 &  9 &  4 &  0
      \end{pmatrix}.
    \]
    The matrix $\Dtwo$ is not a squared Euclidean distance matrix:
    the optimal matrix $A(\gF)$ has a negative eigenvalue $\approx -0.597$,
    so $-\tfrac{1}{2}J\Dtwo J$ is not positive semidefinite.
    
    \textbf{Computing $\gF$.}
    Row means of $\Dtwo$ are $\bar{d}_1 = \tfrac{29}{4}$,
    $\bar{d}_2 = \tfrac{19}{4}$, $\bar{d}_3 = \tfrac{9}{4}$,
    $\bar{d}_4 = \tfrac{29}{4}$, and grand mean $\bar{d}=\tfrac{43}{8}$.
    Formula \eqref{eq:gF} gives
    \[
      \gF
      = \tfrac{1}{16}\bigl(73,\; 33,\; {-7},\; 73\bigr)^{\top},
      \qquad \ones^{\top}\gF = \tfrac{43}{4}.
    \]
    
    \textbf{Optimal matrix.}
    $A(\gF) = -\tfrac{1}{2}J\Dtwo J$ equals
    \[
      A(\gF) = \tfrac{1}{16}
      \begin{pmatrix}
         73 & -19 &   1 & -55 \\
        -19 &  33 &   5 & -19 \\
          1 &   5 &  -7 &   1 \\
        -55 & -19 &   1 &  73
      \end{pmatrix}.
    \]
    Its eigenvalues are approximately
    \[
      8,\qquad 3.347,\qquad 0,\qquad -0.597,
    \]
    giving $\rho^* = \normtwo{A(\gF)} = 8$.
    
    \textbf{Norms and verification.}
    \begin{align*}
      \normF{A(\gF)}  &\approx 8.693, &
      \normtwo{A(\gF)} &= 8\phantom{.000}, &
      \|A(\gF)\|_* &\approx 11.945.
    \end{align*}
    These values are consistent with Theorem~\ref{thm:main}: the same
    matrix $A(\gF)$ minimises the Frobenius norm, the spectral norm, and
    indeed every unitarily invariant norm over the affine family.
    
    \textbf{A non-unique spectral minimiser.}
    Although $\gF$ is the unique Frobenius minimiser, the spectral norm
    need not have a unique minimiser; for the nuclear norm, $\gF$ is the
    unique minimiser by Corollary~\ref{cor:FrobNucUnique}.  
    To illustrate this, set
    \[
      \bg^*
      =
      \gF+\frac{21}{16}\ones .
    \]
    Then
    \[
      A(\bg^*)
      =
      A(\gF)+\frac{21}{16}\ones\ones^\top .
    \]
    Since $A(\gF)\ones=0$, the perturbation
    $\frac{21}{16}\ones\ones^\top$ only affects the one-dimensional
    direction $\operatorname{span}\{\ones\}$ and adds the eigenvalue
    \[
      4\cdot\frac{21}{16}=\frac{21}{4}.
    \]
    Because
    \[
      \frac{21}{4}<8=\rho^*,
    \]
    we still have
    \[
      \normtwo{A(\bg^*)}=8=\normtwo{A(\gF)}.
    \]
    On the other hand,
    \[
      \normF{A(\bg^*)}\approx 10.155
      >
      \normF{A(\gF)}\approx 8.693.
    \]
    Thus this example shows explicitly that the Frobenius minimiser is
    unique, whereas the spectral norm minimisation problem may admit
    additional minimisers.
    \end{example}

\begin{example}[$n=4$, collinear configuration --- rank-$1$ case]
    \label{ex:collinear}
    Let $x_1=0$, $x_2=1$, $x_3=2$, and $x_4=3$ be four equally spaced
    points on the real line, and set
    \[
      \bx=(0,1,2,3)^\top .
    \]
    The corresponding squared distance matrix is
    \[
      \Dtwo =
      \begin{pmatrix}
        0 & 1 & 4 &  9 \\
        1 & 0 & 1 &  4 \\
        4 & 1 & 0 &  1 \\
        9 & 4 & 1 &  0
      \end{pmatrix}.
    \]
    
    \textbf{Computing $\gF$.}
    Formula \eqref{eq:gF} gives
    \[
      \gF = \tfrac{1}{4}(9,1,1,9)^{\top},
      \qquad
      \ones^{\top}\gF = 5 .
    \]
    
    \textbf{Optimal matrix.}
    Since the configuration is one-dimensional and Euclidean, the
    centred matrix satisfies
    \[
      A(\gF)
      =
      -\tfrac{1}{2}J\Dtwo J
      =
      (J\bx)(J\bx)^{\top}.
    \]
    Here
    \[
      J\bx
      =
      \left(-\tfrac{3}{2},-\tfrac{1}{2},
      \tfrac{1}{2},\tfrac{3}{2}\right)^\top,
    \]
    and therefore
    \[
      A(\gF)
      =
      \tfrac{1}{4}
      \begin{pmatrix}
         9 &  3 & -3 & -9 \\
         3 &  1 & -1 & -3 \\
        -3 & -1 &  1 &  3 \\
        -9 & -3 &  3 &  9
      \end{pmatrix}.
    \]
    This matrix has rank exactly $1$ and its only nonzero eigenvalue is
    \[
      \lambda_1=\|J\bx\|^2=5 .
    \]
    
    \textbf{All Schatten norms coincide.}
    By Proposition~\ref{prop:rank1geom},
    \[
      \|A(\gF)\|_{S_p}
      =
      \normtwo{A(\gF)}
      =
      \normF{A(\gF)}
      =
      5,
      \qquad 1\leq p\leq\infty .
    \]
    This illustrates Corollary~\ref{cor:rank}: since
    $\rank(A(\gF))=1$, the Frobenius and spectral minimum values
    coincide, and in fact all Schatten norms take the same value at the
    optimum.
    
    \textbf{Geometric interpretation.}
    The collinear structure means that, after centring, the 
    configuration lies in the one-dimensional subspace of $\ones^\perp$ 
    spanned by $J\bx$.
    Equivalently, the double-centred matrix has rank one:
    \[
      J\Dtwo J
      =
      -2(J\bx)(J\bx)^\top .
    \]
    Thus the unique nonzero eigenvalue of $A(\gF)$ is
    $\|J\bx\|^2=5$, the total centred sum of squares of the
    one-dimensional configuration.
    \end{example}

\begin{example}[$n=5$, norm profile along the direction $\ones$]
\label{ex:profile5}
To illustrate the behaviour of the Frobenius, spectral, and nuclear
norms away from the Frobenius minimiser, consider
\[
  \Dtwo =
  \begin{pmatrix}
     0 &  4 &  9 & 25 & 16 \\
     4 &  0 &  1 &  9 &  4 \\
     9 &  1 &  0 &  4 &  1 \\
    25 &  9 &  4 &  0 &  9 \\
    16 &  4 &  1 &  9 &  0
  \end{pmatrix}.
\]
This matrix is not Euclidean, since $A(\gF)$ has a negative eigenvalue
approximately equal to $-0.887$.  Formula \eqref{eq:gF} gives
\[
  \gF
  \approx
  (7.52,\;0.32,\;-0.28,\;6.12,\;2.72)^\top,
  \qquad
  \ones^\top\gF=16.4 .
\]
The eigenvalues of $A(\gF)$ are approximately
\[
  13.561,\qquad 3.726,\qquad 0,\qquad 0,\qquad -0.887,
\]
and hence
\[
  \rho^*=\normtwo{A(\gF)}\approx 13.561,
  \qquad
  \normF{A(\gF)}\approx 14.091,
  \qquad
  \|A(\gF)\|_*\approx 18.173 .
\]

\textbf{Norm profile.}
We vary $\bg$ along the direction $\ones$ by setting
\[
  \bg(t)=\gF+\frac{t}{5}\ones .
\]
Then
\[
  A(\bg(t))
  =
  A(\gF)+\frac{t}{5}\ones\ones^\top .
\]
Since $A(\gF)\ones=0$, this is a block-diagonal perturbation with
respect to the orthogonal decomposition
\[
  \R^5=\ones^\perp\oplus\operatorname{span}\{\ones\}.
\]
The matrix $\frac{t}{5}\ones\ones^\top$ vanishes on $\ones^\perp$ and
has eigenvalue $t$ in the direction $\ones$.  Therefore the eigenvalues
of $A(\bg(t))$ are those of $A(\gF)$, together with the additional
eigenvalue $t$ replacing the zero eigenvalue in the direction
$\ones$.  Consequently,
\[
  \normtwo{A(\bg(t))}
  =
  \max\{\rho^*,|t|\}.
\]
Table~\ref{tab:normprofile} records the three norms as $t$ varies.

\begin{table}[ht]
\centering
\caption{Norms of $A(\bg(t))$ along the direction
$\bg(t)=\gF+\tfrac{t}{5}\ones$.  Along this direction the spectral
norm is constant for $|t|\leq\rho^*\approx 13.56$.}
\label{tab:normprofile}
\smallskip
\begin{tabular}{rrrr}
\hline
$t$ &
$\normF{A(\bg(t))}$ &
$\normtwo{A(\bg(t))}$ &
$\|A(\bg(t))\|_*$ \\
\hline
$-27.12$ & $30.563$ & $27.121$ & $45.295$ \\
$-13.56$ & $19.556$ & $13.561$ & $31.734$ \\
$ -6.78$ & $15.638$ & $13.561$ & $24.954$ \\
$\phantom{-}0$ & $\mathbf{14.091}$ & $\mathbf{13.561}$ & $\mathbf{18.173}$ \\
$\phantom{-}6.78$ & $15.638$ & $13.561$ & $24.954$ \\
$\phantom{-}13.56$ & $19.556$ & $13.561$ & $31.734$ \\
$\phantom{-}27.12$ & $30.563$ & $27.121$ & $45.295$ \\
\hline
\end{tabular}
\end{table}

The table illustrates three features predicted by the theory along
this one-dimensional slice of the affine family:
\begin{enumerate}
  \item the Frobenius norm has a unique strict minimum at $t=0$,
        that is, at $\bg=\gF$; the nuclear norm also attains its
        unique global minimum at $t=0$;
  \item the spectral norm is flat for $|t|\leq\rho^*$, showing that
        the spectral norm minimisation problem may have non-unique
        minimisers;
  \item all three norms grow symmetrically as $|t|\to\infty$.
\end{enumerate}
Figure~\ref{fig:normprofile} displays the corresponding continuous
profiles, making the flat spectral region and the strict Frobenius
minimum simultaneously visible.
\end{example}

\begin{figure}[ht]
    \centering
    \includegraphics[width=0.82\textwidth]{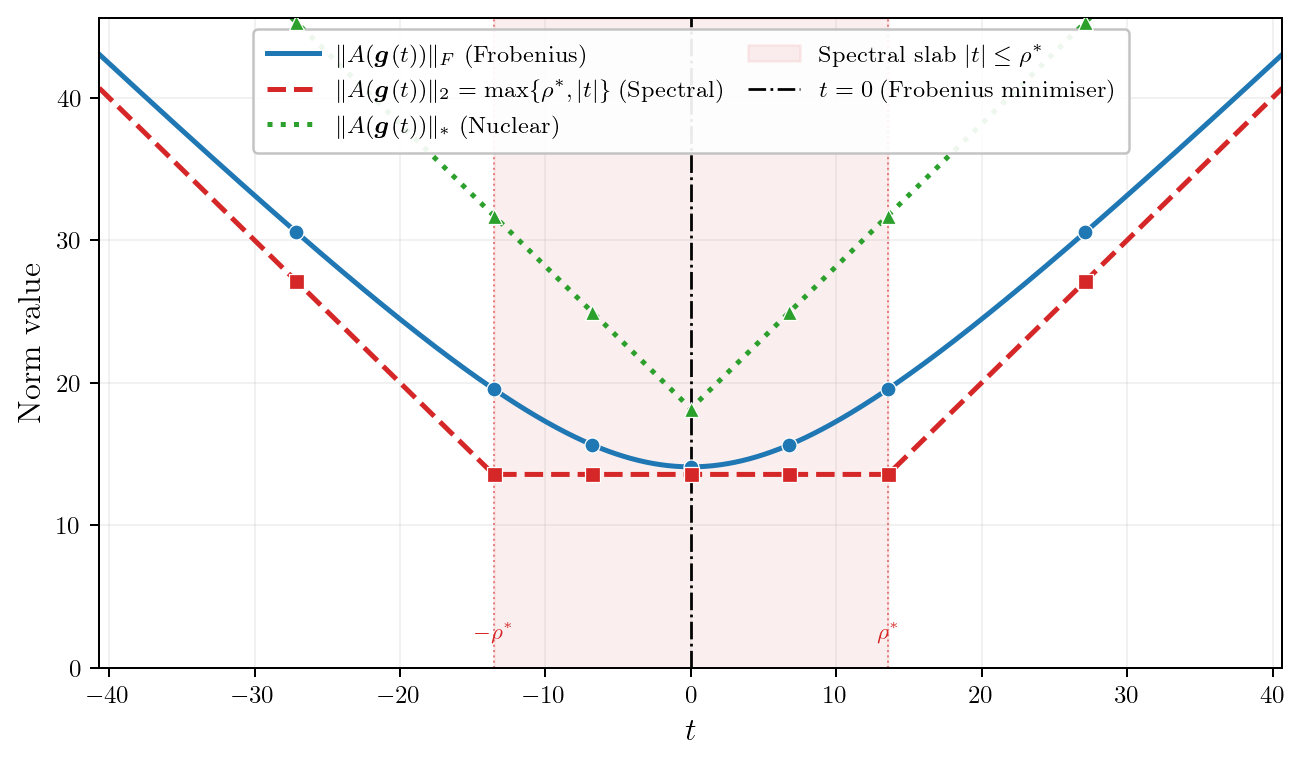}
    \caption{Norm profiles for Example~\ref{ex:profile5} along the
direction $\bg(t)=\gF+\tfrac{t}{5}\ones$.
Along this direction one has
$A(\bg(t))=A(\gF)+\tfrac{t}{5}\ones\ones^\top$,
so the spectral norm satisfies the exact formula
$\normtwo{A(\bg(t))}=\max\{\rho^*,|t|\}$ (dashed red),
which is constant for $|t|\leq\rho^*\approx 13.56$ (shaded region).
The Frobenius norm $\normF{A(\bg(t))}$ (solid blue) has a unique
strict minimum at $t=0$, i.e.\ at $\bg=\gF$; the
nuclear norm $\|A(\bg(t))\|_*$ (dotted green) also attains its
minimum along this slice at $t=0$.
Filled symbols ($\bullet$\,circle, $\blacksquare$\,square,
$\blacktriangle$\,triangle) mark the seven entries of
Table~\ref{tab:normprofile}.}
    \label{fig:normprofile}
    \end{figure}

\section{Discussion and extensions}
\label{sec:discussion}

\paragraph{Summary}
We have proved that, for the affine family~\eqref{eq:Ag} of symmetric
matrices compatible with prescribed off-diagonal squared dissimilarity
data $\Dtwo$, the Frobenius minimiser $\gF$ simultaneously minimises
every unitarily invariant norm (Theorem~\ref{thm:main}).  The key
structural fact is the centred compression identity
\[
  JA(\bg)J=-\tfrac{1}{2}J\Dtwo J,
  \qquad \bg\in\R^n,
\]
together with the observation that the Frobenius minimiser is the
unique parameter satisfying $A(\gF)\ones=0$.  Consequently,
\[
  A(\gF)=-\tfrac{1}{2}J\Dtwo J,
\]
so the optimal representative is precisely the double-centred squared
dissimilarity matrix that appears in CMDS and, in applied fields such as
numerical ecology and microbiome analysis, in PCoA~\cite{Torgerson1952,Gower1966,Legendre2012,Satten2017}.

This yields a purely variational characterisation of the canonical
CMDS\slash PCoA matrix, with no Euclidean realisability required.

\paragraph{Uniqueness and non-uniqueness}
The Frobenius minimiser $\gF$ is unique by strict convexity of the
Frobenius objective.  For a general unitarily invariant norm,
Theorem~\ref{thm:main} shows that $\gF$ is always a minimiser, but it
does not by itself imply uniqueness.  Corollary~\ref{cor:strict_unique}
gives uniqueness for strictly convex unitarily invariant norms
(Schatten-$p$ with $1<p<\infty$, including the Frobenius norm), and
Corollary~\ref{cor:FrobNucUnique} gives uniqueness for the nuclear norm.
For other non-strictly convex norms such as the spectral norm or the
Ky Fan $k$-norms ($1\leq k<n$), additional minimisers may exist.  This
is already visible
along the one-dimensional slice
\[
  \bg(t)=\gF+\frac{t}{n}\ones,
\]
for which
\[
  A(\bg(t))=A(\gF)+\frac{t}{n}\ones\ones^\top.
\]
Since $A(\gF)\ones=0$, this perturbation only adds the eigenvalue $t$
in the direction $\ones$.  Hence, along this slice,
\[
  \normtwo{A(\bg(t))}
  =
  \max\{\normtwo{A(\gF)},|t|\},
\]
so the spectral norm is constant for
$|t|\leq\normtwo{A(\gF)}$.  This slice is the case $\bv=0$ of
Theorem~\ref{thm:spectral_set}, which describes the whole set of
spectral minimisers.  For the Ky Fan $k$-norms with $1<k<n$ the
corresponding description is not available: it would require analysing
the equality cases in the pinching inequality for those gauges.

\paragraph{Rank condition and collinearity}
Proposition~\ref{prop:rank1geom} characterises the degenerate case
$\rank(J\Dtwo J)\leq 1$ in equivalent algebraic and geometric terms.
In the Euclidean setting this condition is equivalent to the
configuration having affine dimension at most one, that is, to the
points being collinear.  In that case,
\[
  A(\gF)=-\tfrac{1}{2}J\Dtwo J
\]
is a positive semidefinite matrix of rank at most one.  If the
configuration is non-trivial and one-dimensional, then
\[
  A(\gF)=(J\bx)(J\bx)^\top,
\]
and all Schatten norms of $A(\gF)$ coincide with
$\|J\bx\|^2$, the total centred sum of squares of the scalar
configuration.

\paragraph{Open question: equality cases and solution sets}
The proof of Theorem~\ref{thm:main} uses the contractivity of
orthogonal pinching for unitarily invariant norms.  It therefore
identifies a universal minimiser, and the equality cases decide how
many others there are.  These cases are now settled for three families
of norms.  For the Schatten-$p$ norms with $1<p<\infty$ and for the
nuclear norm, $\gF$ is the unique minimiser
(Corollaries~\ref{cor:FrobNucUnique} and~\ref{cor:strict_unique}).  For
the spectral norm, Theorem~\ref{thm:spectral_set} describes the whole
minimising set, which is a compact convex set reducing to $\{\gF\}$
exactly when $A(\gF)=0$.

What remains open is the intermediate range: for the Ky Fan $k$-norms
with $1<k<n$, and for other non-strictly convex gauges not covered
above, the structure of the solution set may depend on the
multiplicities of the singular values, and we do not know a
characterisation analogous to Theorem~\ref{thm:spectral_set}.  The
difficulty is that the argument used for the spectral norm rests on the
equivalence between $\normtwo{S}\leq\rho$ and the positive
semidefiniteness of $\rho I\pm S$, which has no direct counterpart for
a partial sum of singular values.

\paragraph{Relation to sensor network localization and SDP}
Semidefinite formulations of EDM completion, such as that of
Alfakih et al.\ \cite{Alfakih1999}, optimise over centred
positive-semidefinite Gram matrices whose induced distances agree
with, or approximate, partially specified data; sensor network
localisation is a leading application~\cite{Krislock2012}.  These
problems share the classical Gram-to-distance map with the present
setting, but they are not obtained merely by imposing
$A(\bg)\succeq 0$ on our norm-selection problem, since the
variables, the fixed data and the objective differ.  The present
paper shows that the unconstrained variational selection, obtained
by minimising any unitarily invariant norm, always selects the same
canonical representative
\[
  A(\gF)=-\tfrac{1}{2}J\Dtwo J .
\]
If $A(\gF)\succeq 0$, this representative is also feasible for the
semidefinite constraints of EDM completion; if not, it remains the
unconstrained optimum but is not feasible for those semidefinite
formulations.

\paragraph{Computational cost}
The explicit formula \eqref{eq:gF} computes $\gF$ in $O(n^2)$
operations: one pass over $\Dtwo$ to compute row sums and one to
compute the grand sum.  This cost is optimal at the level of dense
input matrices, since reading $\Dtwo$ already requires $O(n^2)$
operations.  The optimal matrix
\[
  A(\gF)=-\tfrac{1}{2}J\Dtwo J
\]
can then be formed in $O(n^2)$ additional operations using the usual
double-centring formula, without explicitly forming the matrix $J$.

\section*{Acknowledgements}
This work was supported by the
Uni\-ver\-si\-dad CEU Car\-de\-nal Herrera under grants INDI25\slash{}17 and GIR25\slash{}14.

\section*{Author contributions}
\textbf{M.~Nuria de las Heras Santos}: Conceptualization, Methodology,
Formal analysis, Investigation, Writing---original draft,
Writing---review \& editing.
\textbf{Antonio Falc\'o}: Conceptualization, Methodology,
Formal analysis, Investigation, Writing---original draft,
Writing---review \& editing.
\textbf{Francisco Javier Mu\~{n}oz Almaraz}: Conceptualization, Methodology,
Formal analysis, Investigation, Writing---original draft,
Writing---review \& editing.
All authors have read and agreed to the published version of the manuscript.

\section*{Declaration of generative AI and AI-assisted technologies in the manuscript preparation process}

During the preparation of this work the authors used Claude (Anthropic) for
language editing and formatting assistance. After using this AI tool,
the authors reviewed and edited the content as needed and take full
responsibility for the content of the manuscript.

\bibliographystyle{elsarticle-num}

\end{document}